\documentclass[oneside,12pt,reqno]{amsart}
\usepackage{amssymb,amsmath,amsthm,bbm,enumerate,mdwlist,url,multirow,hyperref,amsthm}
\usepackage[pdftex]{graphicx}
\usepackage[shortlabels]{enumitem}
\usepackage[T1]{fontenc}

\theoremstyle{definition}
\newtheorem{definition}{Definition}%Extra square-bracket argument achieves that the numbering is the same as for definition (single uniform counter). 
\theoremstyle{theorem}
\newtheorem{proposition}[definition]{Proposition}

\newtheorem{theorem}[definition]{Theorem}
\newtheorem{corollary}[definition]{Corollary}

\numberwithin{equation}{section}
\numberwithin{definition}{section}
\theoremstyle{remark}
\newtheorem{remark}[definition]{Remark}

\def\PP{\mathsf{P}}
\def\GG{\mathfrak{G}}

\def\ZZ{\mathbb{Z}}
\def\NN{\mathbb{N}}
\def\NN0{\mathbb{N}_0}

\def\dd{\mathrm{d}}

\def\i{{\rm i}}
\def\D{{\rm d}}

\newmuskip\pFqmuskip

\newcommand*\pFq[6][8]{%
  \begingroup % only local assignments
  \pFqmuskip=#1mu\relax
  \mathcode`\,=\string"8000
  \begingroup\lccode`\~=`\,
  \lowercase{\endgroup\let~}\pFqcomma
  {}_{#2}F_{#3}{\left[\genfrac..{0pt}{}{#4}{#5};#6\right]}%
  \endgroup
}
\newcommand{\pFqcomma}{\mskip\pFqmuskip}

\begin{document}
\title[Double hypergeometric L\'evy processes]{Double hypergeometric L\'evy processes and self-similarity}
\author{Andreas E. Kyprianou}
\address{Andreas E. Kyprianou, Department of Mathematical Sciences, University of Bath, Claverton Down, Bath, BA2 7AY, UK.}
\email{a.kyprianou@bath.ac.uk}
\author{Juan Carlos Pardo}
\address{Juan Carlos Pardo, Apartado Postal 402, CP 36000, 
Calle Jalisco s/n, Mineral de Valencianam
Guanajuato, Gto. Mexico}
\email{jcpardo@cimat.mx}
\author{Matija Vidmar$^*$}
\address{Matija Vidmar, Department of Mathematics, University of Ljubljana, and Institute of Mathematics, Physics and Mechanics, Slovenia. $^{*}$Corresponding author.}
\email{matija.vidmar@fmf.uni-lj.si}
\begin{abstract} Motivated by a recent paper of Budd \cite{budd}, where a new family of positive self-similar Markov processes  associated to stable processes appears,  we introduce a new family of L\'evy processes, called the double hypergeometric class, whose Wiener-Hopf factorisation is explicit, and as a result many functionals can be determined in closed form.
\end{abstract}

\keywords{L\'evy processes; Wiener-Hopf factorisation; self-similar Markov processes; complete Bernstein functions; Pick functions}

\subjclass[2010]{60G51} 

\maketitle

\section{Introduction}
\noindent Real-valued L\'evy processes are the class of stochastic processes which have stationary and independent increments  and paths that are right-continuous with left-limits.  We may include in this definition the possibility that the L\'evy process is killed and sent to a cemetery state, in which case the killing time must necessarily take the form of  an independent and exponentially distributed random time. As a consequence of this definition alone, one can proceed to show that they are strong Markov processes whose transitions are entirely characterised by a single quantity, the characteristic exponent. That is to say, if $(\xi_t,t\geq 0)$ is a one-dimensional L\'evy process with probabilities $(\mathbf{P}_x,x\in\mathbb{R})$, then, for all $t\geq 0$ and $\theta\in \mathbb{R}$, there exists a function $\Psi: \mathbb{R}\to\mathbb{C}$ such that 
\[
\mathbf{E}_0[{\rm e}^{\i \theta \xi_t}] = {\rm e}^{-\Psi(\theta)t}, \qquad t\geq 0,
\]
where $\Psi$ can be described by  the so-called L\'evy-Khintchine formula, i.e. 
\[
\Psi \left( {z} \right) =q+\mathrm{i}a {z} +\frac{1}{2}   \sigma^2 z^2+\int_{\mathbb{R}}(1-\mathrm{e}^{\mathrm{i}{z}  x}+\mathrm{i}{z}  x\mathbf{1}%
_{(|x|<1)})\,\Pi ({\D}x), \qquad z\in \mathbb{R},
\]
with  $q\geq 0$,
$a\in \mathbb{%
R}$, $\sigma^2\geq 0$  and $\Pi $ is a measure (called the L\'evy measure) concentrated on
$\mathbb{R}\backslash \{0\}$ satisfying $\int_{\mathbb{R}}(
1\wedge x^{2}) \Pi ({\D}x)<\infty$. Here, $q$ is the rate of the aforementioned exponential killing time, $-a$ is linear drift, $\sigma:=\sqrt{\sigma^2}$ is the Gaussian coefficient and $\Pi(\D x)$ determines the rate at which jumps of size $x\in\mathbb{R}\backslash\{0\}$ occur.
\smallskip

\noindent For each non-zero characteristic exponent (as described above) there exists a unique factorisation (the so-called spatial Wiener--Hopf factorisation) such that 
\begin{equation}
\Psi(\theta) = \Phi^+(-\i\theta) \Phi^-(\i\theta), \qquad \theta\in\mathbb{R},
\label{WHF}
\end{equation}
where  $\Phi^+, \Phi^-$ are Bernstein functions or, in other words, the Laplace exponents of subordinators;  uniqueness being up to multiplying and dividing $\Phi^+$ and $\Phi^-$, respectively, with a strictly positive constant. The factorisation has a physical interpretation in the sense that one may identify $ \Phi^+(\lambda) = - t^{-1}\log \mathbf{E}[{\rm e}^{-\lambda S^+_t}]$, $\lambda\geq 0$, $t>0$, where $(S^+_t, t\geq 0)$,  the increasing ladder heights process, is a subordinator whose range agrees with that of $(\sup_{s\leq t}\xi_s, t\geq 0)$. In parallel to the uniqueness of $\Phi^+$ only up to premultiplication by a strictly positive constant, $S^+$ is only meaningful up to a scaling of the time-axis. The role of $\Phi^-$ is similar to that of $\Phi^+$ albeit in relation to the dual process $-\xi$.
\smallskip

\noindent Vigon's theory of philanthropy \cite{vigon} is arguably a milestone in the theory of  L\'evy processes in that it provides precise conditions through which one may fashion a L\'evy process by defining it through a selected pair $\Phi^+, \Phi^-$,  as opposed to e.g. its transition law, L\'evy measure or characteristic exponent. Following the appearance of a number of new Wiener--Hopf factorisations in \cite{bertoin-yor, cc, lamperti-stable}, this observation was made in \cite{hg, hg-bis}  and explored further  in \cite{meromorphic, ehg, hk}.  
\smallskip

\noindent In particular it is argued in \cite{hg-bis} that a natural choice of Bernstein functions (equivalently subordinators) to fulfill the roles of $\Phi^+, \Phi^-$ are those of the so-called {\it beta class} and take the form 
%\[
\begin{equation}\label{betaclass}
\Phi(\lambda) = \frac{\Gamma(\lambda+\beta+\gamma)}{\Gamma(\lambda+\beta)}, \qquad \lambda \geq 0,
\end{equation}
%\]
where $0\leq \alpha\leq \beta+\gamma$ and $\gamma\in [0,1]$. When selecting $\Phi^+$ and $\Phi^-$ from the class of beta subordinators within the restrictions permitted by Vigon's theory of philanthropy, the resulting family of L\'evy processes was called the {\it hypergeometric class}. The given name of this class reflects the fact that the associated L\'evy measure of the L\'evy process fulfilling \eqref{WHF}  can be expressed in a quite simple way using the Gauss hypergeometric function (often written  ${_2F_1}$). 
\smallskip

\noindent Selecting $\Phi^+$ and $\Phi^-$ from the class of beta subordinators
not only brings about a degree of tractability to a number of associated problems, but it also brings the hypergeometric class in line with a number of other attractive theories such as that of L\'evy processes with completely monotone Wiener--Hopf factors; cf. \cite{rogers}. As alluded to above, a broader definition for the hypergeometric class was also explored in later papers.
\smallskip

\noindent Our objective here is to consider yet another deepening of the definition of the hypergeometric class of L\'evy processes.
As such, we are interested in understanding the extent to which one may select the factors $\Phi^+, \Phi^-$ from the class of Bernstein functions  taking the form
\begin{equation}
\Phi(\lambda) = \mathtt{B}(\alpha, \beta, \gamma, \delta;\lambda):=\frac{\Gamma(\lambda+\alpha)\Gamma(\lambda+\beta)}{\Gamma(\lambda+\gamma)\Gamma(\lambda+\delta)}, \qquad \lambda \geq 0,
\label{doublebeta}
\end{equation}
for an appropriate configuration of constants $\alpha, \beta, \gamma,\delta$. This will lead to a class of L\'evy processes that we will call the {\it the double hypergeometric class}. More generally, our work will show that it is in principle possible to select the factors $\Phi^+, \Phi^-$ from a class of Bernstein functions taking the form 
\[
\Phi(\lambda) = \prod_{i =1}^n\frac{\Gamma(\lambda+\alpha_i)}{\Gamma(\lambda+\gamma_i)}, \qquad \lambda \geq 0,
\]
again, for appropriate configurations of constants $\alpha_1,\ldots, \alpha_n$ and 
$\gamma_1,\ldots, \gamma_n$ and $n\in\mathbb{N}$. That said, we will stick to the setting of \eqref{doublebeta}.
\smallskip

\noindent The generalisation we pursue in this paper is motivated through a specific example of a new L\'evy process that emerged recently in the literature concerning planar maps whose characteristic exponent is represented as a product of two functions of the form \eqref{doublebeta} for a suitable choice of parameters. In that setting, \cite{budd} introduces a path transformation of a two-sided jumping stable process in which, with a point of issue in $(0,\infty)$, each crossing of the origin is `ricocheted' back into the positive half-line with a positive probability and otherwise killed. The resulting process is a positive self-similar Markov process (pssMp) and accordingly enjoys a representation as a space-time changed L\'evy process via the so-called Lamperti transform. When examining the Lamperti-transform of the ricocheted stable process, for appropriate parameter choices, a double hypergeometric L\'evy process emerges. Then a natural question arises \textit{is the form of the characteristic exponent  of such L\'evy process a proper Wiener-Hopf factorisation}? We make these comments more precise later in this paper at the point that we deal with our second main objective. The latter is to examine the interaction of double hypergeometric L\'evy processes and path transformations of the stable process, both in the context of pssMp but also real-valued self-similar Markov processes.

\smallskip

\noindent The rest of the paper is organised as follows. In the next section we introduce the class of double hypergeometric L\'evy processes and explore their advantages in terms of the Wiener--Hopf factorisation and their integrated exponential functionals.  In Section \ref{section:applications} we briefly remind the reader of how self-similar Markov processes are related to L\'evy processes or, more generally, Markov additive processes. This then allows us in Section~\ref{section:ricocheted} to discuss  how the class of double hypergeometric L\'evy processes appears naturally in the context of  certain types of self-similar Markov processes that are built from a path perturbation of a stable process. The final Section~\ref{section:proofs} is concerned with the proofs of the main results. 

\section{Double beta subordinators and double hypergeometric L\'evy processes}
\noindent Let us start by introducing some notation. As already alluded to above,  ${_2F_1}$ stands for the Gauss hypergeometric function. However, we will, more generally, use ${_pF_q}$ for the generalised hypergeometric function: $$
\pFq{p}{q}{a_1\cdots a_p}{b_1\cdots b_q}{z}
%{_pF_q}(a_1;\cdots;a_p;b_1;\cdots;b_q;z):
=\sum _{k=0}^{\infty } \frac{\left(a_1\right)_k\ldots \left(a_p\right)_k}{\left(b_1\right)_k\ldots \left(b_q\right)_k} \frac{z^k}{k!},
$$ where $(a)_k=a(a+1)\cdots(a+k-1)$ denotes the rising factorial power and is frequently called the Pochhammer symbol in the theory of special functions. %We will sometimes write integrals/expectations in ``Markov theory parlance'', viz. $\PP[Z]=\EE_\PP[Z]$, $\PP[Z;A]=\EE_\PP[Z\mathbbm{1}_A]$ etc.
\smallskip

\noindent Next, suppose that  $A$ and $B$ are two countably infinite subsets of $\mathbb{R}$, bounded from above and discrete. Then $B$ is said to  interlace with $A$ if $\max A\geq \max B$ and if for the order-preserving enumerations $(a_n)_{n\in \ZZ_{\leq 0}}$ and $(b_n)_{n\in \ZZ_{\leq 0}}$ of $A$ and $B$, respectively, one has that $a_{n-1}<b_n<a_n$ for all $n\in \ZZ_{\leq 0}$, the nonpositive integers. If $A$ and $B$ interlace, then necessarily they are disjoint.

\smallskip

\noindent Below we give our first  key result of this paper. It identifies a class of subordinators having Laplace transform of the form \eqref{doublebeta}. Before stating it, recall that (i) a real meromorphic function is a meromorphic function that maps the real line excluding the poles into the real line, (ii) a Pick (or Nevanlinna-Pick) function is an analytic map defined on a subset of $\mathbb{C}$ containing the open upper complex half-plane and that maps the latter into itself \cite[Definition~6.7]{bernstein}, and (iii) the non-constant complete Bernstein functions are precisely the Pick functions that are nonnegative on $(0,\infty)$ \cite[Theorem~6.9]{bernstein} (the latter omits the qualification ``non-constant'', but without it, it is clearly false). 
\begin{theorem}\label{proposition}
Let $\{\alpha,\beta,\gamma,\delta\}\subset (0,\infty)$, assume $\{\alpha-\beta,\gamma-\delta,\gamma-\alpha,\delta-\alpha, \gamma-\beta,\delta-\beta\}\cap \ZZ=\emptyset$, and suppose that  $\{-\alpha,-\beta\}+\ZZ_{\leq 0}$ interlaces with $\{-\gamma,-\delta\}+\ZZ_{\leq 0}$, which, when without loss of generality $\alpha< \beta$ and $\gamma< \delta$, is equivalent to assuming that
 either
\begin{enumerate}[(I)]
 \item\label{charact:i} there is a $k\in \mathbb{N}_0$ with $\gamma+k< \alpha+k< \delta< \beta< \gamma+k+1$, or
\item\label{charact:ii}  there is a $k\in \mathbb{N}$ with $\alpha+k-1< \gamma+k< \beta< \delta< \alpha+k$.
\end{enumerate}
 Then: 
 
 {\tt a)}  $\gamma+\delta<\alpha+\beta<\gamma+\delta+1$; 
 
   {\tt b)} $\mathtt{B}(\alpha, \beta, \gamma, \delta;\cdot)$ defined in \eqref{doublebeta} 
  is the Laplace exponent of a killed, infinite activity, pure jump  subordinator with cemetery state $\infty$, whose L\'evy measure possesses a completely monotone density $f:(0,\infty)\to [0,\infty)$ given by
\begin{align}\label{eq:density}
f(s)&=-\frac{{\rm e}^{-\alpha s}\Gamma(\beta-\alpha)}{\Gamma(\gamma-\alpha)\Gamma(\delta-\alpha)}
\pFq{2}{1}{1+\alpha-\gamma,1+\alpha-\delta}{1+\alpha-\beta}{{\rm e}^{-s}}\notag\\
%{_2F_1}(1+\alpha-\gamma,1+\alpha-\delta;1+\alpha-\beta;{\rm e}^{-s})\notag\\
&\hspace{2cm}-\frac{{\rm e}^{-\beta s}\Gamma(\alpha-\beta)}{\Gamma(\gamma-\beta)\Gamma(\delta-\beta)}
\pFq{2}{1}{1+\beta-\gamma,1+\beta-\delta}{1+\beta-\alpha}{{\rm e}^{-s}}, \quad s\in (0,\infty);
%{_2F_1}(1+\beta-\gamma,1+\beta-\delta;1+\beta-\alpha;{\rm e}^{-s}). 
\end{align}

{\tt c)} $\mathtt{B}(\alpha, \beta, \gamma, \delta;\cdot)$ is a non-constant complete  Bernstein function, indeed it is a real meromorphic Pick function; 

{\tt d)} the associated potential measure $u:(0,\infty)\to [0,\infty)$ (whose Laplace transform is given by $1/\mathtt{B}(\alpha, \beta, \gamma, \delta;\cdot)$)  also admits a density which is completely monotone, and it is given by
\begin{align}\label{eq:potential}
u(x)&=
\frac{{\rm e}^{-\gamma x}\Gamma(\delta-\gamma)}{\Gamma(\alpha-\gamma)\Gamma(\beta-\gamma)}
\pFq{2}{1}{1-\alpha+\gamma,1-\beta+\gamma}{1-\delta+\gamma}{{\rm e}^{-x}}\notag\\
%{_2F_1}(1-\alpha+\gamma,1-\beta+\gamma;1-\delta+\gamma;{\rm e}^{-x})\\
&\hspace{2cm}+\frac{{\rm e}^{-\delta x}\Gamma(\gamma-\delta)}{\Gamma(\alpha-\delta)\Gamma(\beta-\delta)}
\pFq{2}{1}{1-\alpha+\delta,1-\beta+\delta}{1-\gamma+\delta}{{\rm e}^{-x}}, \quad x\in (0,\infty). 
%{_2F_1}(1-\alpha+\delta,1-\beta+\delta;1-\gamma+\delta;{\rm e}^{-x}),
\end{align}

\smallskip

\noindent Conversely, when, ceteris paribus, the interlacing property of $\{-\alpha,-\beta\}+\ZZ_{\leq 0}$ with $\{-\gamma,-\delta\}+\ZZ_{\leq 0}$ fails, then the function $\mathtt{B}(\alpha, \beta, \gamma, \delta;\cdot)$ in \eqref{doublebeta}  is not Pick, and so is not a non-constant complete Bernstein function on $(0,\infty)$. 
\smallskip

\noindent Finally, suppose merely $\{\alpha,\beta,\gamma,\delta\}\subset [0,\infty)$, and  either 
\begin{enumerate}[(i)]
 \item\label{charact:i'} there is a $k\in \mathbb{N}_0$ with $\gamma+k\leq\alpha+k\leq\delta\leq\beta\leq\gamma+k+1$, or
\item\label{charact:ii'}  there is a $k\in \mathbb{N}$ with $\alpha+k-1\leq\gamma+k\leq\beta\leq\delta\leq\alpha+k$, or
\item\label{charact:iii'} one of the preceding holds once the substitutions  $\alpha\leftrightarrow \beta$ and/or $\gamma\leftrightarrow\delta$ have been effected.
\end{enumerate}
Then  $\mathtt{B}(\alpha, \beta, \gamma, \delta;\cdot)$  of  \eqref{doublebeta} is still a complete Bernstein function.
\end{theorem}

\noindent We say that a subordinator $S$ having the Laplace exponent $\mathtt{B}(\alpha, \beta, \gamma, \delta;\cdot)$ as in \eqref{doublebeta} is a {\it double beta subordinator}.  The class of quadruples $(\alpha,\beta,\gamma,\delta)\in [0,\infty)^4$ satisfying the conditions \ref{charact:i'}-\ref{charact:ii'}-\ref{charact:iii'}  of  Theorem \ref{proposition} is denoted $\GG$.  Note,  we obtain the beta subordinators of \eqref{betaclass} as a special case of double beta subordinators in the class   $\GG$ by taking  $\alpha=\gamma=0$ and substituting $\beta\rightarrow \beta+\gamma $, $\delta\rightarrow \beta$ in  \eqref{doublebeta} (it satisfies  \ref{charact:i'} with $k=0$). The class of those  quadruples $(\alpha,\beta,\gamma,\delta)\in (0,\infty)^4$ that satisfy the same conditions  \ref{charact:i'}-\ref{charact:ii'}-\ref{charact:iii'} except that $\leq$ is replaced by $<$ is denoted $\GG^\circ$; it identifies \emph{precisely} those quadruples $(\alpha,\beta,\gamma,\delta)\in (0,\infty)^4$ for which $\{\alpha-\beta,\gamma-\delta,\gamma-\alpha,\delta-\alpha, \gamma-\beta,\delta-\beta\}\cap \ZZ=\emptyset$ and that make  $\mathtt{B}(\alpha, \beta, \gamma, \delta;\cdot)$  of  \eqref{doublebeta} into a non-constant complete Bernstein function.

\smallskip 
\noindent We should point out that our focus in Theorem~\ref{proposition} on double beta subordinators corresponding to \eqref{doublebeta} with $(\alpha,\beta,\gamma,\delta)\in \GG^\circ$, rather than more generally $(\alpha,\beta,\gamma,\delta)\in \GG$, is a matter of convenience: to describe the resulting subordinators for all the possible constellations that $\GG$ admits would be prohibitive in scope. For instance, as already remarked, $\GG$ includes the beta subordinators, but also the trivial (zero) subordinators killed at an indpendent exponential random time, etc. Besides, the conditions ``$\alpha,\beta,\gamma,\delta>0$ and  $\{\alpha-\beta,\gamma-\delta,\gamma-\alpha,\delta-\alpha, \gamma-\beta,\delta-\beta\}\cap \ZZ=\emptyset$'' serve to ensure that there are no ``cancellations or reductions'' of the gamma factors in \eqref{doublebeta}, i.e. $\mathtt{B}(\alpha, \beta, \gamma, \delta;\cdot)$ is ``truly'' a quotient of four gammas. On the other hand, our insistance on the interlacing property, beyond merely demanding that $\alpha,\beta,\gamma,\delta>0$ and  $\{\alpha-\beta,\gamma-\delta,\gamma-\alpha,\delta-\alpha, \gamma-\beta,\delta-\beta\}\cap \ZZ=\emptyset$, i.e. on the  ``non-constant complete Bernstein'' property, is more than just convenience: it is not clear to us how to establish the Bernstein property of   $\mathtt{B}(\alpha, \beta, \gamma, \delta;\cdot)$  from  \eqref{doublebeta} when the complete Bernstein property fails (if this can in fact happen, see the comments below Proposition \ref{thm:wh-for-pssmp-ricocheted}). \\%, cf. Remark~\ref{remark:not-complete-Bernstein?}).

\noindent From Vigon's theory of philanthropy \cite{vigon}, and in particular the fact that double beta subordinators corresponding to $\GG$ are completely monotone, thereby having associated L\'evy measures that are absolutely continuous with non-increasing densities,  we can now define a new class of L\'evy processes from the double beta subordinator class as follows; see Section 6.6. of \cite{kyprianou}.

\begin{corollary}\label{corollary}
Let $\{(\alpha,\beta,\gamma,\delta),(\hat{\alpha},\hat{\beta},\hat{\gamma},\hat{\delta})\}\subset \GG$. %Set $\Phi^+(\cdot)=\mathtt{B}(\alpha, \beta, \gamma, \delta;\cdot)$ and ${\Phi}^-(\cdot)=\mathtt{B}(\hat\alpha, \hat\beta, \hat\gamma, \hat\delta;\cdot)$. 
Then there exists a L\'evy process with characteristic exponent $\Psi$ that respects the spatial Wiener--Hopf factorisation in the form 
\[
\pushQED{\qed} 
\Psi(\theta)=\mathtt{B}(\alpha, \beta, \gamma, \delta;-i\theta)\mathtt{B}(\hat\alpha, \hat\beta, \hat\gamma, \hat\delta; i\theta), \qquad \theta\in\mathbb{R}.\qedhere
\popQED
\]
 \label{WHFcor}
\end{corollary}

\noindent In light of the above, we call a L\'evy process, $\xi$, of the type introduced in the preceding corollary a {\it double hypergeometric L\'evy process}. %Strictly speaking, we could introduce a 9th parameter in the definition of a double hypergeometric L\'evy process by noting that, if $\xi\sim \mathrm{DH}^{\alpha,\beta;\hat{\alpha},\hat{\beta}}_{\gamma,\delta;\hat{\gamma},\hat{\delta}}$, then $c\xi$ does not belong to $\mathrm{DH}^{\alpha,\beta;\hat{\alpha},\hat{\beta}}_{\gamma,\delta;\hat{\gamma},\hat{\delta}}$. However, $x\xi$ has characteristic exponent which takes the same form  as that of $\xi$ albet with $\theta$ replaced by $c\theta$.
 The following proposition gathers some basic properties of our new class of double hypergeometric L\'evy processes (and accordingly justifies the name ``double hypergeometric'').
\begin{proposition}\label{doubledensity}
Let $\{(\alpha,\beta,\gamma,\delta),(\hat{\alpha},\hat{\beta},\hat{\gamma},\hat{\delta})\}\subset \GG$ and $\xi$ be a double hypergeometric L\'evy process with these parameters as in Corollary~\ref{corollary}. We have the following assertions. 
\begin{enumerate}[(i)]
\item\label{further:characteristics} Assume even $\{(\alpha,\beta,\gamma,\delta),(\hat{\alpha},\hat{\beta},\hat{\gamma},\hat{\delta})\}\subset \GG^\circ$.  Then $\xi$, once stripped away of its killing, is a meromorphic L\'evy process in the terminology of \cite{meromorphic}. The L\'evy measure $\Pi$ of $\xi$ has a density $\pi$ that is given by 
\begin{align*}
\pi(x) =& -\frac{\Gamma (\alpha + \hat{\alpha})\Gamma(\alpha +\hat{\beta}) \Gamma(\beta-\alpha){\rm e}^{-\alpha x} }{\Gamma (\alpha + \hat{\gamma}) \Gamma(\alpha + \hat{\delta}) \Gamma(\gamma-\alpha)\Gamma(\delta-\alpha)}
%{\rm B}^{\hat{\alpha},\hat{\beta}}_{\hat{\gamma},\hat{\delta}}(\alpha)
\pFq{4}{3}{\alpha+\hat{\alpha},\alpha+\hat{\beta},1+\alpha-\gamma,1+\alpha-\delta}{1+\alpha-\beta,\alpha+\hat{\gamma},\alpha+ \hat{\delta}}{{\rm e}^{-x}}\\
%{_4F_3}(\alpha+\hat{\alpha},\alpha+\hat{\beta},1+\alpha-\gamma,1+\alpha-\delta;1+\alpha-\beta,\alpha+\hat{\gamma},\alpha+ \hat{\delta};{\rm e}^{-x})\\
&
-\frac{\Gamma (\beta + \hat{\beta})\Gamma(\beta +\hat{\alpha}) \Gamma(\alpha-\beta){\rm e}^{-\beta x} }{\Gamma (\beta + \hat{\gamma}) \Gamma(\beta + \hat{\delta}) \Gamma(\gamma-\beta)\Gamma(\delta-\beta)}
%{\rm B}^{\hat{\beta},\hat{\beta}}_{\hat{\gamma},\hat{\delta}}(\beta)
\pFq{4}{3}{\beta+\hat{\beta},\beta+\hat{\alpha},1+\beta-\gamma,1+\beta-\delta}{1+\beta-\alpha,\beta+\hat{\gamma},\beta+ \hat{\delta}}{{\rm e}^{-x}}
%[\alpha\leftrightarrow \beta],
\end{align*}
for $x\in (0,\infty)$. For $x\in (-\infty,0)$ it is the same, except that the quantities with and without a hat get interchanged for their respective counterparts and $-x$ replaces $x$.
\item\label{further:characteristics:Gaussian} There is  a Gaussian component iff $\alpha+\beta=\gamma+\delta+1$ and $\hat{\alpha}+\hat{\beta}=\hat{\gamma}+\hat{\delta}+1$, in which case the diffusion coefficient $\sigma^2=2$.
\item\label{further:characteristics:lifetime}  Recall that we may assume without loss of generality that  $\gamma\leq\delta$, $\alpha\leq\beta$,  $\hat{\gamma}\leq\hat{\delta}$, $\hat{\alpha}\leq\hat{\beta}$ and assume (for simplicity\footnote{If $\delta=0$ ($\hat{\delta}=0$), then automatically $\alpha=\gamma=0$ ($\hat{\alpha}=\hat{\gamma}=0$).}) that $\delta\hat{\delta}>0$. Then $\xi$ has infinite lifetime iff: either $ \gamma=\hat{\gamma}=0<\alpha\hat{\alpha}$,  in which case $\xi$ oscillates; or
$\gamma=0<\hat{\gamma}\alpha$, in which case $\xi$ drifts to $\infty$; or
 $\hat{\gamma}=0<\gamma\hat{\alpha}$, in which case $\xi$ drifts to $-\infty$. 
 
\end{enumerate}
\end{proposition}
\noindent The exclusion of the parameters corresponding to $\GG\backslash \GG^\circ$ in the statement of \ref{further:characteristics} is not without cause: for instance the functions $1$ and $\mathrm{id}_{[0,\infty)}$ are of the double beta class, but they do not yield a meromorphic L\'evy process, at least not in the sense of \cite{meromorphic}; similarly the expression of \ref{further:characteristics} for the L\'evy density does not hold true across the whole class of double hypergeometric L\'evy processes correspoding to $\GG$.

\smallskip

\noindent In the next section, we shall shortly discuss  examples of how the double hypergeometric L\'evy process emerges in the setting of certain self-similar Markov processes. As we will see in the exposition there, for the theory of the latter processes, an important quantity of interest is the  exponential functional of L\'evy processes. Below we show that it is possible to characterise the exponential functional of double hypergeometric L\'evy processes via an explicit Mellin transform.
Our method is based on the ``verification'' result of \cite[Subsection~8.1]{hg-bis}. In the following, we use $\mathbf{P}$ to denote the law of the double hypergeometric L\'evy process.
%For the ``guessing'' part we are guided by \cite[Section~8]{budd} where this is already done for some subset of the double hypergeometric class (a similar programme  is effected in \cite[Section~3]{ehg} for the extended hypergeometric class of L\'evy processes, following \cite{hg-bis} where it was done for the ``basic''  hypergeometric class).

\begin{proposition}\label{proposition:mellin}
Let $\{(\alpha,\beta,\gamma,\delta),(\hat{\alpha},\hat{\beta},\hat{\gamma},\hat{\delta})\}\subset \GG$ and   $\xi$ be double hypergeometric L\'evy process with these parameters as in Corollary~\ref{corollary}.   Assume $\hat{\gamma}\leq \hat{\delta}$, $\hat{\alpha}\leq \hat{\beta}$ and $0<\hat{\gamma}<\hat{\alpha}$ (in fact the very last condition is automatic if even $(\hat{\alpha},\hat{\beta},\hat{\gamma},\hat{\delta})\in \GG^\circ$, while the first two are without loss of generality).  Then, for a given ${c}\in (0,\infty)$, with $\zeta$ being the lifetime of $\xi$, for $s\in \mathbb{C}$ with $\Re(s)\in (0,1+\hat{\gamma}{c})$, %the exponential functional $$I(\xi/{c}):=\int_0^\zeta {\rm e}^{-\xi_t/{c}}dt,$$ where $\zeta$ is the lifetime of $\xi$, satisfies
\footnotesize
\begin{equation}\label{eq:mellin}
\mathbf{E}\left[\left(\int_0^\zeta {\rm e}^{-\xi_t/{c}}dt\right)^{s-1}\right]=C\Gamma(s)\frac{G({c} \gamma+s;{c})G({c} \delta+s;{c})}{G({c}\alpha+s;{c})G({c} \beta+s;{c})}\frac{G(1+{c}\hat{\alpha}-s;{c})G(1+{c} \hat{\beta}-s;{c})}{G(1+{c} \hat{\gamma}-s;{c})G(1+{c} \hat{\delta}-s;{c})},
\end{equation}\normalsize
where $C$ is a normalization constant such that the left- and right-hand side agree at  $s= 1$ (i.e. both are equal to unity), and $G$ is Barnes' double gamma function (see \cite[p. 121]{hg-bis},  and further references therein,  for its definition and basic properties).
\end{proposition}
\begin{remark}\label{remark:mellin}
In fact in the proof only the functional form of the characteristic exponent of $\xi$, that $\xi$ is killed or else drifting to $\infty$, and the inequality $\gamma+\delta-\alpha-\beta+\hat{\alpha}+\hat{\beta}-\hat{\gamma}-\hat{\delta}<6$ will be used. Hence, if it is known a priori that a possibly killed L\'evy process has the Laplace exponent as given in Corollary~\ref{corollary} with $\gamma+\delta-\alpha-\beta+\hat{\alpha}+\hat{\beta}-\hat{\gamma}-\hat{\delta}<6$, and that it is killed or drifts to $\infty$, then the condition ``$\{(\alpha,\beta,\gamma,\delta),(\hat{\alpha},\hat{\beta},\hat{\gamma},\hat{\delta})\}\subset \GG$''  can be dropped in favor of all of the coefficients $\alpha,\ldots,\hat{\delta}$ being just nonnegative. 
\end{remark}
%\begin{remark}\label{remark:mellin}
%In fact in the proof only the functional form of the Laplace exponent of $X$, that $X$ is killed or else drifting to $\infty$, and the inequality $\gamma+\delta-\alpha-\beta+\hat{\alpha}+\hat{\beta}-\hat{\gamma}-\hat{\delta}<6$ will be used. Hence, if it is known a priori that a possibly killed L\'evy process has the Laplace exponent as given in Theorem~\ref{corollary} with $\gamma+\delta-\alpha-\beta+\hat{\alpha}+\hat{\beta}-\hat{\gamma}-\hat{\delta}<6$, and that it is killed or drifts to $\infty$, then the condition ``$\{(\alpha,\beta,\gamma,\delta),(\hat{\alpha},\hat{\beta},\hat{\gamma},\hat{\delta})\}\subset \GG$''  can be dropped in favor of all of the coefficients $\alpha,\ldots,\hat{\delta}$ being just nonnegative. 
%\end{remark}

\noindent As indicated in the introduction, before moving to the proofs of the main results in this section, we will first look at how double hypergeometric L\'evy processes appear naturally in the setting of a special family of self-similar Markov processes. Our first step in this direction is to briefly introduce the latter processes.

\section{Self-similar Markov processes}\label{section:applications}

\noindent Let $Y=(Y_s, s\geq 0)$ be a positive self-similar Markov process (pssMp) with self-similarity index $\alpha\in (0,\infty)$ and  probabilities $(\PP_y, y>0)$. Then $Y$ enjoys a bijection with the class of L\'evy processes (as presented in the introduction of this article) via the Lamperti transform.  Lamperti's result gives a bijection between the class of pssMps and the class of L\'evy processes, possibly killed at an independent exponential time with cemetery state $-\infty$, such that, under $\PP_y$, $y>0$, with the convention $\xi_\infty=-\infty$,
%s: old:
%\begin{align*}
%	Z_t=\exp(\xi_{\varphi_t}),\quad  t< T_{0}:= \inf\{t\geq 0\colon Z_t=0\},
%\end{align*}
%where $\xi_0=\log(z)$ and $\varphi(t)=\int_0^t \exp(\alpha \xi_s)ds$.
\begin{align}
\label{pssMpLamperti}
	Y_t=\exp(\xi_{\varphi_t}),\qquad t\in [0,\infty),%t< I_\infty : = \int_0^\infty \exp(\alpha\xi_u)\D u,
\end{align}
where $\varphi_t=\inf\{s>0: \int_0^s \exp(\alpha \xi_u)\D u >t \}$ and the L\'evy process $\xi$  is started in $\log y$.

\smallskip

\noindent In recent work, e.g. \cite{GV, VG, Kiu, Chy-Lam, rivero, KKPW, doring}, effort has been invested into extending  the theory of pssMp to the setting of $\mathbb{R}$.
 Analogously to Lamperti's representation, for a real self-similar Markov process (rssMp), say $Y$, there is a Markov additive process $((\xi_t,J_t), t\geq 0)$ on $\mathbb{R}\times \{-1,1\}$ such that, again under the convention $\xi_\infty=-\infty$,
\begin{align}\label{LK}
   Y_t =  J_{\varphi_t}\exp\bigl( \xi_{\varphi_t}\bigr) ,\qquad t\in [0,\infty),%\quad  t< I_\infty : = \int_0^\infty e^{\alpha\xi_s}ds,
\end{align}
%\begin{align}\label{LK}
 %  Z_t = z \exp\bigl( \xi_{\varphi^{-1}(t|z|^{-\alpha})}\bigr) \,J_{\varphi^{-1}(t|z|^{-\alpha})},\quad  0 \le t < \zeta, 
%\end{align}
where $\varphi_t=\inf\{s>0 : \int_0^s \exp(\alpha \xi_u){\D} u >t\}$ and $(\xi_0,J_0)=(\log |Y_0|,[Y_0])$ with
$$
[z]=\begin{cases} 1 &\mbox{ if } z>0, \\-1 &\mbox{ if }z<0.\end{cases}
$$
The representation \eqref{LK} is known as the Lamperti--Kiu transform. 
Here, by  Markov additive process (MAP), we mean  the regular strong Markov process with probabilities ${\bf P}_{x,i}$, $x\in\mathbb{R}$, $i\in\{-1,1\}$, such  that $(J_t, t\geq 0)$ is a continuous time Markov chain on $\{-1,1\}$ (called the modulating chain) and, for any $i\in \{-1,1\}$ and $t\geq 0$,
\begin{align*}
	&\text{given }\{J_t=i\},\text{ the pair }(\xi_{t+s}-\xi_t,J_{t+s})_{s\geq 0} \text{ is independent of the past up to $t$}\\%\mathcal F_t\\
	&\qquad\text{ and has the same distribution as }(\xi_s, J_s)_{s\geq 0}\text{ under }{\bf P}_{0,i}.
\end{align*}
If the MAP is killed, then $\xi$ is sent  to the cemetery state $-\infty$. All background results for MAPs that relate to the present article can be found in the appendix of Dereich et al. \cite{doring}.\smallskip

\smallskip

\noindent The analogue of the characteristic exponent of a L\'evy process for the above MAP is provided by the matrix-valued function $\mathbf{\Psi}$ such that  for all $t\in [0,\infty)$, $\{i,j\}\subset \{1,-1\}$, 
\[
\mathbf{E}_{0,i}[{\rm e}^{\i\theta \xi_t};J_t=j]=({\rm e}^{\mathbf{\Psi}(\theta)t})_{ij},\qquad \theta\in \mathbb{R},
\]
and 
\[
\mathbf{\Psi}(\theta)= - \mathrm{diag}(\Psi_1(\theta),\Psi_{-1}(\theta))+Q\circ G(\theta),
\]
where: $Q$ is the generator matrix of $J$; $G$ is the matrix of the characteristic functions of the extra jumps that are inserted into the path of the MAP $\xi$ when $J$ transitions, the diagonal elements of  $G$ being set to unity; $\Psi_1$ and $\Psi_{-1}$ are the characteristic exponents of the L\'evy processes that govern the evolution of the MAP component $\xi$ when it is in states $1$ and $-1$, respectively. The symbol $\circ$ denoted element-wise multiplication (Hadamard product). %Let us also denote by $U_{1,-1}^*$ the law of the size of the extra jump inserted into the path of the MAP $\xi$ when $J$ transitions between $1$ and $-1$; analogously for $U_{-1,1}^*$.  

\smallskip

\noindent The mechanism behind the Lamperti--Kiu representation is thus simple. The modulation $J$ governs the sign of $X$ and, on intervals of time for which there is no change in sign, the Lamperti--Kiu representation effectively plays the role of the  Lamperti representation of a pssMp, multiplied with $-1$ if $X$ is negative. In a sense, the MAP formalism gives a concatenation of signed Lamperti representations between times of sign change.
	Typically one can assume the Markov chain $J$ to be irreducible as otherwise the corresponding self-similar Markov process only switches signs at most once and can therefore be treated using the theory of  pssMp.

\smallskip

%\noindent \textcolor{red}{Many examples of  pssMp and }

\section{Richocheted stable processes}\label{section:ricocheted}

\noindent
Let us now turn to the setting of stable processes and path perturbations thereof that fall under the class of self-similar Markov processes and give rise to double hypergeometric L\'evy processes in their Lamperti representation. 
\bigskip

\noindent Let  $(X_t,t\geq 0)$ with probabilities $(\mathbb{P}_x,x\in \mathbb{R})$ (as usual $\mathbb{P}:=\mathbb{P}_0$)  be a \textit{strictly $\alpha$-stable process} (henceforth just written `stable process'), that is to say
an unkilled L\'evy process which
also satisfies the \textit{scaling property}: under $\mathbb{P}$,
for every $c > 0$,
the process $(cX_{t c^{-\alpha}})_{t \ge 0}$
has the same law as $X$.
It is known that $\alpha$ necessarily belongs to $(0,2]$, and the case $\alpha = 2$
corresponds to Brownian motion, which we exclude.
The L\'evy-Khintchine representation of such a process
is as follows:
$\sigma = 0$, 
$\Pi$ is absolutely continuous with density given by 
\[ 
 \pi(x): =  c_+ x^{-(\alpha+1)} \mathbf{1}_{(x > 0)} + c_- {|x|}^{-(\alpha+1)} \mathbf{1}_{(x < 0)},
\qquad  x \in \mathbb{R},
\]
where $c_+,\, c_- \ge 0$.

\smallskip

\noindent
For consistency with the literature that we shall appeal to in this article,
we shall always parametrise our $\alpha$-stable process such that 
\[ c_+ = \Gamma(\alpha+1) \frac{\sin(\pi \alpha \rho)}{\pi} \quad \text{and} \quad
  c_- = \Gamma(\alpha+1) \frac{\sin(\pi \alpha \hat\rho)}{\pi },
  \]
where
$\rho = \mathbb{P}(X_t \ge 0)$ is the positivity parameter, 
and $\hat\rho = 1-\rho$.  Note when $\alpha=1$, then $c_+=c_-$. Moreover,   we may also identify the exponent as taking the form 
\begin{equation}\label{Psi_alpha_rho_parameterization}
\Psi(\theta) = |\theta|^\alpha ({\rm e}^{\pi\i\alpha(\frac{1}{2} -\rho)} \mathbf{1}_{(\theta>0)} + {\rm e}^{-\pi\i\alpha (\frac{1}{2} - \rho)}\mathbf{1}_{(\theta<0)}), \qquad \theta\in\mathbb{R}.
\end{equation}
 
 \smallskip

 \noindent With this normalisation, we take the point of view that the class of stable processes is
parametrised by $\alpha$ and $\rho$; the reader will note that
all the quantities above can be written in terms of these parameters.
We shall restrict ourselves a little further within this class
by excluding the possibility of having only one-sided jumps. In particular, this rules out the possibility that $X$ is a subordinator or the negative of a subordinator, which occurs when $\alpha\in(0,1)$ and either $\rho =1$ or $0$. 
 \smallskip

 \noindent For this class of two-sided jumping stable processes, it is well known that, when  $\alpha\in(0,1]$, the process never hits the origin (irrespective of the point of issue) and $\lim_{t\to\infty}|X_t| = \infty$. When $\alpha\in(1,2)$, the process hits the origin with probability one (for all points of issue). In both cases, the process will cross the origin infinitely often.

\smallskip

 \noindent In \cite[Section~6]{budd} the following example of  a positive self-similar Markov processes, say $Y^*=(Y^*_t,t\geq 0)$, called a ricocheted stable process,  was introduced in the context of the theory of planar maps. 
Fix a  $\mathfrak{p}\in [0,1]$ and let $\tau^-_0 = \inf\{t>0: X_t<0\}$. The stochastic dynamics of $Y^*$ are as follows. 
From its point of issue in $(0,\infty)$, $Y^*$ evolves as a $X$ until its first passage into $(-\infty,0)$, i.e. $\tau^-_0$.
At that time an independent coin is flipped with probability $\mathfrak{p}$ of heads. If heads is thrown, then  the process $Y^*$ is immediately transported to $-X_{\tau^-_0}$ (i.e. it is ``ricocheted'' across $0$). If tails is thrown, then it is sent to $0$ and the process is killed. In the event  $Y^*$ is  ricocheted, it continues to evolve as an independent copy of $X$, flipping a new coin on first pass into $(-\infty,0)$ and so on. 

\smallskip

 \noindent The process $Y^*$ can be viewed as a way of resurrecting (with probability $\mathfrak{p}$) the process $(X_t\mathbf{1}_{(t<\tau^-_0)}, t\geq 0)$ every time it is about to die; indeed, it degenerates to the latter process when $\mathfrak{p}=0$. In this respect the ricocheted process is related to the forthcoming work of \cite{CCRtba}.
This paper highlights the interesting phenomenon as to whether such resurrected processes will eventually continuously absorb at the origin or not. 

\smallskip

\noindent As alluded to above, ricocheted processes were considered in  \cite{budd}. In particular it was shown in Proposition 11 there that the characteristic exponent of the L\'evy process associated to $Y^*$ via the Lamperti transform takes the shape
\begin{equation}\label{eq:budd}
\Psi^*(\theta)=\frac{\Gamma(\alpha-\i\theta)\Gamma(1+\i\theta)}{\pi}\left[\sin(\pi(\alpha\hat{\rho} -\i\theta))-\mathfrak{p}\sin(\pi\alpha\hat{\rho})\right],\quad \theta\in \mathbb{R}.
\end{equation}
Moreover,  by setting 
\[
\sigma:=\frac{1}{2}-\alpha\hat{\rho}\in \left(-\frac{1}{2},\frac{1}{2}\right)\text{ and }b:=\frac{1}{\pi}\arccos(\mathfrak{p}\cos(\pi\sigma))\in \left[\vert\sigma\vert,\frac{1}{2}\right],
\]
 we may also write %according to \cite[Proposition~1,1]{budd} 
\begin{equation}\label{eq:WH}
\Psi^*(\theta)%=\log \EE[\exp(z\xi_1^*);1<\zeta]
=2^{\alpha}\dfrac{\Gamma \left(\frac{1+\i\theta}{2}\right)\Gamma\left(\frac{2+\i\theta}{2}\right)}{\Gamma\left( \frac{\sigma+b+\i\theta}{2}\right)\Gamma\left(\frac{\sigma-b+\i\theta+2}{2}\right)}\times\frac{\Gamma \left(\frac{\alpha-\i\theta}{2}\right)\Gamma\left(\frac{1+\alpha-\i\theta}{2}\right)}{\Gamma\left( \frac{b-\sigma-\i\theta}{2}\right)\Gamma\left( \frac{2-\sigma-b-\i\theta}{2}\right)},\quad \theta\in \mathbb{R}.
\end{equation}
With this in hand we now see an apparent relation with  Corollary~\ref{corollary}.%using Theorem \ref{proposition} (taking account of the comment immediately below Corollary \ref{corollary}).

\begin{proposition}\label{thm:wh-for-pssmp-ricocheted}
Provided $\alpha\in [1-\sigma-b,b-\sigma+1]$, equivalently $\mathfrak{p}\sin(\pi\alpha\hat{\rho})\leq \sin(\pi\alpha\rho)$, then the two  factors either side of $\times$ in \eqref{eq:WH} identify the Wiener-Hopf factorisation. That is to say, the exponent in \eqref{eq:WH} belongs to the class of double hypergeometric L\'evy processes. \qed
\end{proposition}

\noindent There are a number of remarks that are worth making at this point. First, 
$b\in \{\sigma,-\sigma\}$ occurs iff $\mathfrak{p}=1$. In that case we see that $\xi^*$ has infinite lifetime; moreover, by computing $\Psi^{*\prime}(0)$, we see that $Y^*$ is absorbed at the origin in finite time, oscillates between $0$ and $\infty$, or drifts to $\infty$, according as $\sigma<0$, $\sigma=0$ or $\sigma>0$ (equivalently $\alpha\hat\rho>1/2$,  $\alpha\hat\rho=1/2$ or  $\alpha\hat\rho<1/2$), which resonates with the work on resurrected L\'evy processes in \cite{CCRtba}.

\smallskip

\noindent
Second,
when $\mathfrak{p}=0$, and hence $b=\frac{1}{2}$, we recover the Wiener-Hopf factorization of the Lamperti stable L\'evy process corresponding to killing the stable process on hitting $(-\infty,0)$ (in this case the condition $\alpha\in [1-\sigma-b,b-\sigma+1]$ is met).%\cite[Theorem~5.3]{watson}.

\smallskip

\noindent
Third, one can easily verify that  the condition $\alpha\in [1-\sigma-b,b-\sigma+1]$ is a necessarily and sufficient condition needed  to ensure the ``interlacing'' property of Theorem \ref{proposition}. Despite the exponent $\Psi^*$ falling outside of the double hypergeometric class when the aforementioned condition fails, we nonetheless conjecture that  \eqref{eq:WH} is in fact the correct factorization for the whole of the parameter regime. By way of example, taking e.g. $\alpha=21/20$, $\rho=5/21$ and $\mathfrak{p}=0.9$, then the condition  $\alpha\in [1-\sigma-b,b-\sigma+1]$ fails, and indeed numerically it is seen that the second factor in  \eqref{eq:WH} does not correspond to a Pick function; still, an inspection of the first few derivatives suggests that it nevertheless corresponds to a Bernstein function.

\smallskip

\noindent Finally, we note that, in the spirit of \cite{cc} and \cite{budd} there appear two more characteristic exponents of L\'evy processes which are obtained by a simple Esscher transform once we note that $\i(b+\sigma)$ and $-\i(b-\sigma)$ are roots of $\Psi^*$.  Indeed, again provided $\alpha\in [1-\sigma-b,b-\sigma+1]$, we also have that $\Psi ^*(z+\i(b+\sigma))$ and $\Psi ^*(z-\i(b-\sigma))$ fall into the double hypergeometric class.

\smallskip
\noindent
 Let now $T_0^*$ be the first hitting time of $0$ by the process $Y^*$,  and let $\zeta^*$ be the lifetime of $\xi^*$. By the Lamperti transform we may  express $$T_0^*=\int_0^{\zeta^*}e^{\alpha \xi_s^*}\dd s=2^{-\alpha}\int_0^{2^\alpha\zeta^*}e^{-(-2\xi^*_{2^{-\alpha}u})/(2/\alpha)}\dd u,$$ so that Proposition~\ref{proposition:mellin} (coupled with Remark~\ref{remark:mellin}) yields the Mellin transform of $T_0^*$ provided $0<b-\sigma<\alpha$. We leave making the eventual expression explicit to the interested reader. %r, noting only that the condition $0<b-\sigma$, i.e. $\mathfrak{p}<1$ or $\alpha\hat{\rho}> \frac{1}{2}$, merely ensures that $\xi^*$ is either killed or drifts to $-\infty$, which is equivalent to having $T_0^*$  finite (with a positive probability, and then a.s.). 
 We do however give the following detail.
 
 \begin{corollary}\label{corollary:law-of-sup-and-of-inf}
Assume $\alpha\in [1-\sigma-b,b-\sigma+1]$. Let, respectively, $\overline{Y^*}_{T_0^*-}:=\sup\{Y^*_s:s\in [0,T_0^*)\}$ and $\underline{Y}^*_{T_0^*-}:=\inf\{Y^*_s:s\in [0,T_0^*)\}$ be the overall supremum  and the overall infimum of $Y^*$ before absorption at zero. Correspondingly  introduce $\overline{\xi^*}_{\zeta^*-}$ and $\underline{\xi^*}_{\zeta^*-}$ in the obvious way. We have the following assertions: 
\begin{enumerate}[(i)]
\item\label{cor:i:sup}  If $\mathfrak{p}=1$ and $\alpha\hat{\rho}\leq \frac{1}{2}$, then a.s. $T_0^*=\zeta^*=\overline{Y^*}_{T_0^*-}=\overline{\xi^*}_{\zeta^*-}=\infty$. Otherwise the quantities $T_0^*$ and $\overline{Y^*}_{T_0^*-}=\exp(\overline{\xi^*}_{\zeta^*-})$ are a.s. finite, and the Laplace transform of the law of $\log(\overline{Y^*}_{T_0^*-}/Y_0)=\overline{\xi^*}_{\zeta^*-}-\xi^*_0$ (this law being also that of  $\log(Y^*_{T_0^*-}/\underline{Y^*}_{T_0^*-})=\xi^*_{\zeta-}-\underline{\xi^*}_{\zeta^*-}$ when $\mathfrak{p}<1$) is given by the map 
\[
[0,\infty)\ni z\mapsto \frac{\sqrt{\pi}\Gamma(\alpha)2^{1-\alpha}}{\Gamma\left( \frac{b-\sigma}{2}\right)\Gamma\left( \frac{2-\sigma-b}{2}\right)}\frac{\Gamma\left( \frac{b-\sigma+z}{2}\right)\Gamma\left( \frac{2-\sigma-b+z}{2}\right)}{\Gamma \left(\frac{\alpha+z}{2}\right)\Gamma\left(\frac{1+\alpha+z}{2}\right)}.
\]
\item\label{cor:ii:inf} If $\mathfrak{p}=1$ and $\alpha\hat{\rho}\geq \frac{1}{2}$, then $\zeta^*=\infty$ and  $\underline{Y}^*_{T_0^*-}=0$. Otherwise $\underline{Y}^*_{T_0^*-}=\exp(\underline{\xi^*}_{\zeta^*-})$ is a.s. strictly positive, and the Laplace transform of the law of $\log(Y_0/\underline{Y}^*_{T_0^*-})=\xi^*_0-\underline{\xi^*}_{\zeta^*-}$ (this law being also that of  $\log(\overline{Y^*}_{T_0^*-}/Y^*_{T_0^*-})=\overline{\xi^*}_{\zeta^*-}-\xi^*_{\zeta^*-}$ when $\mathfrak{p}<1$) is given by the map 
\[
[0,\infty)\ni z\mapsto \frac{\sqrt{\pi}}{\Gamma\left( \frac{\sigma+b}{2}\right)\Gamma\left(\frac{\sigma-b+2}{2}\right)}\frac{\Gamma\left( \frac{\sigma+b+z}{2}\right)\Gamma\left(\frac{\sigma-b+z+2}{2}\right)}{\Gamma \left(\frac{1+z}{2}\right)\Gamma\left(\frac{2+z}{2}\right)}.
\]
 %If $\hat{\kappa}^*(0)=0$, then a.s. $\zeta^*=\infty$, $T_0^*<\infty$ and  $\underline{Y}^*_{T_0^*-}=0$.
\end{enumerate}
\end{corollary}
\begin{proof}
The various (a.s.; we have in some places omitted this qualification for brevity) equalities of random variables follow from the Lamperti transformation. The equalities in law are part of the statement of the Wiener-Hopf factorization of $\xi^*$.  To see \ref{cor:i:sup}, let $\eta^*$ be the lifetime of the increasing ladder heights subordinator $H^{+*}$ of $\xi^*$ and denote by $\Phi^{+*}$ the Laplace exponent of $H^{+*}$. Then $\eta^*\sim\mathrm{Exp}(\Phi^{+*}(0))$ and $H^{+*}_{\eta^*-}=\overline{\xi^*}_{\zeta^*-}$. When $\Phi^{+*}(0)=0$, i.e. $b=\sigma$, this renders $\zeta^*=\overline{\xi^*}_{\zeta^*-}=\infty$ a.s., hence by the Lamperti transform a.s. $T_0^*=\overline{\xi^*}_{\zeta^*-}=\infty$. Otherwise, when $\Phi^{+*}(0)>0$, it then follows from the statements surrounding the Wiener-Hopf factorization of L\'evy processes that $\mathbf{E}[e^{-\theta H^{+*}_{\zeta^*-}}]=\frac{\Phi^{+*}(0)}{\Phi^{+*}(\theta)}$ for $\theta\in [0,\infty)$. The argument for  \ref{cor:ii:inf} is similar.
\end{proof}

\smallskip

\noindent We can also push the notion of a ricocheted stable process into the realms of rssMp but before doing so, we present another interesting example of a pssMp due to Alex Watson \cite{watson-private}, $Y^\natural=(Y^\natural_t,t\ge 0)$, associated to the symmetric stable process $X$, and  whose corresponding L\'evy process in the Lamperti transform also lies in the double hypergeometric class. Indeed, fix a  $\mathfrak{q}\in [0,1)$, take $\rho=1/2$ and let $\tau^-_0 $ be as before. The stochastic dynamics of $Y^\natural$ are as follows. 
From its point of issue in $(0,\infty)$, $Y^\natural$ evolves as  $X$ until its first passage into $(-\infty,0)$.
At that time an independent coin is flipped with probability $\mathfrak{q}$ of heads. If heads is thrown, then  the process $Y^\natural$ is immediately transported to $X_{\sigma_0}$, where $\sigma_0=\inf\{s\ge \tau_0: X_s>0\}$
(i.e. it is ``glued'' to its next positive position, the negative part being thus ``censored away''). If tails is thrown, then it is sent to $0$ and the process is killed. In the event  $Y^\natural$ is  glued, it continues to evolve as an independent copy of $X$, flipping a new coin on first pass into $(-\infty,0)$ and so on. Had we allowed $\mathfrak{q}=1$, then this case would correspond to the censored stable process of \cite{Watson-censored}.

\smallskip 
\noindent The characteristic exponent of $Y^\natural$, denote it by $\Psi^\natural$, may then be identified as:
\begin{equation}\label{psiglued}
\Psi^\natural(\theta)=\frac{\Gamma(\frac{\alpha}{2}-i\theta)\Gamma(\alpha-\i\theta)}{\Gamma(-\gamma+\frac{\alpha}{2}-\i\theta)\Gamma(\gamma+\frac{\alpha}{2}-\i\theta)}\times\frac{\Gamma(1-\frac{\alpha}{2}+i\theta)\Gamma(1+\i\theta)}{\Gamma(1-\gamma-\frac{\alpha}{2}+\i\theta)\Gamma(1+\gamma-\frac{\alpha}{2}+\i\theta)}, \quad \theta\in \mathbb{R},
\end{equation}
where $\gamma:=\frac{1}{\pi}\arcsin(\sqrt{\mathfrak{q}}\sin(\pi\frac{\alpha}{2}))$; see the comments below the proof of  Proposition \ref{proposition:rssMp}. By noting that $\gamma<\frac{\alpha}{2}\land (1-\frac{\alpha}{2})$, we can check that the two factors either side of the sign $\times$ identify the Wiener-Hopf factors and that $Y^\natural$ belongs to the double hypergeometric class.
\smallskip

\noindent As indicated above, we can introduce also a real-valued version, $X^*$, of the ricocheted stable process. Indeed, we can define the process $X^*$ similarly to $Y^*$ but with some slight differences. On crossing the origin from $(0,\infty)$ to $(-\infty,0)$, independently with probability $\mathfrak{p}$ it is ricocheted and with probability $1-\mathfrak{p}$, rather than being killed,  it continues to evolve with the dynamics of a stable process. Additionally, as the state space of $X^*$ is $\mathbb{R}$, when passing from $(-\infty, 0)$ to $(0,\infty)$, the same path adjustment is made on the independent flip of a coin, albeit with a different probability $\hat{\mathfrak{p}}$. Note, in the special case that $\mathfrak{p}=\hat{\mathfrak{p}}=0$, we have that $X^*$ is equal in law to nothing more than the underlying stable process $X$ from which its paths are derived. The reader should also bear in mind that we do not assume any a priori relation between $\mathfrak{p}$ and $\hat{\mathfrak{p}}$: we stress this because we use the notation $\hat{\rho}$ for $1-\rho$, so one might be misled into thinking that $\hat{\mathfrak{p}}=1-\mathfrak{p}$, which however is \emph{not} being assumed.\label{stress}

\smallskip

\noindent It is worth spending a little bit of time to address the question as to why the processes $X^*$, $Y^\natural$ and $Y^*$ are indeed  self-similar Markov processes (ssMp). As alluded to earlier, the fact that $Y^*$ is a ssMp comes from the reasoning of \cite{budd}. The explanation given there was via the use of the associated infinitesimal generators, which, in principle needs a little care with the  rigorous association  of the generator to the appropriate semigroup of $Y^*$. We argue here that there is a relatively straightforward pathwise justification that comes directly out of the Lamperti transform for $Y^*$.

\begin{proposition}\label{proposition:rssMp}We have that  
$Y^*$, $Y^\natural$ and $X^*$ are all ssMp with with self-similarity index $\alpha$.
\end{proposition}
\begin{proof}
Our strategy will be simply to identify the pathwise definition of $Y^*$, $Y^\natural$ and $X^*$  directly with a Lamperti/Lamperti--Kiu decomposition. Once this has been done, the statement follows immediately from the bijection onto the class of pssMp/rssMp in the Lamperti/Lamperti--Kiu transform.

\smallskip

\noindent Let us start with the case of $Y^*$. From \cite{cc} it is known that $(X_t\mathbf{1}_{(t<\tau^-_0)}, t\geq 0)$ is a pssMp, moreover the L\'evy process that underlies its Lamperti transform, say $\xi^\dagger$ has an explicit form for its characteristic exponent (indeed it belongs to the hypergeometric L\'evy processes). As a L\'evy process it falls into the category which are killed at an independent and exponentially distributed time with  a strictly positive rate, say $q^\dagger$. Suppose we write $\Psi^\dagger$ for the characteristic exponent of $\xi^\dagger$ and  consider the process $\xi^0$ whose characteristic exponent is $\Psi^\dagger - q^\dagger$. That is to say, $\xi^0$ is the L\'evy process whose dynamics are those of $\xi^\dagger$ albeit the exponential killing is removed.

\smallskip

\noindent Next, define a compound Poisson process $\xi^\circ$ with arrival rate $q^\dagger \mathfrak{p}$ and jump distribution which is equal in distribution to the random variable $\log (|X_{\tau^-_0}|/ X_{\tau^-_0-})$. Note, straightforward scaling arguments (of both the stopping time $\tau^-_0$ and $X$) can be used to show that  the distribution of  $\log (|X_{\tau^-_0}|/ X_{\tau^-_0-})$ does not depend on the point of issue of $X$; see for example Chapter 13 of \cite{kyprianou}.
The characteristic exponent of this compound Poisson process  is given by $q^\dagger \mathfrak{p}(1 - \mathbb{E}_1[(|X_{\tau^-_0}|/ X_{\tau^-_0-})^{\i\theta}])$.
\smallskip

\noindent Finally, we build the L\'evy process $\xi$ to have characteristic exponent given by 
\begin{equation}
\Psi^*(\theta) =\Psi^\dagger(\theta) - q^\dagger + q^\dagger \mathfrak{p}(1 - \mathbb{E}_1[(|X_{\tau^-_0}|/ X_{\tau^-_0-})^{\i\theta}]) + 
q^\dagger (1-\mathfrak{p}).
\label{psistart}
\end{equation}
In other words, the independent sum of $\xi^0$, the L\'evy process $\xi^\dagger$ without killing, and the compound Poisson process described in the previous paragraph with the resulting stochastic process killed at an independent and exponentially distributed time with rate $q^\dagger (1-\mathfrak{p})$.

\smallskip

\noindent Now consider the Lamperti transform of $\xi$. By splitting the resulting path over the events corresponding to the arrival of points in the compound Poisson processes and the killing time, it is straightforward to see that the stochastic evolution of \eqref{pssMpLamperti} matches that of $Y^*$ and hence they are equal in law. In other words, $Y^*$ is a positive self-similar Markov process. 

\smallskip

\noindent  In the above description, instead of killing at rate $q^\dagger (1-\mathfrak{p})$, we could also  see this as the rate at which switching of an independent Markov chain occurs from $+1$ to $-1$. Accordingly to build up a MAP which corresponds to the process $X^*$, we need to define 
\[
G_{1,-1}^*(\theta): = \mathbb{E}_1[(|X_{\tau^-_0}|/ X_{\tau^-_0-})^{\i\theta}]=\frac{\Gamma(\alpha-\i\theta)\Gamma(1+\i \theta)}{\Gamma(\alpha)},\quad z\in \mathbb{R},\qquad \theta \geq 0,
\]
where the second equality comes from \cite[Corollary~11]{rivero}.
Note  that, although we have seen the Fourier transform of $\log (|X_{\tau^-_0}|/ X_{\tau^-_0-})$ before, this is also the change in the radial exponent $\xi$ when $X^*$ passes from $(0,\infty)$ to $(-\infty,0)$.
\smallskip

\noindent  By replacing the role of $\rho$ by $\hat\rho$ in the definition of $\Psi^\dagger, q^\dagger, \mathfrak{p}$ (henceforth written  $\hat{\Psi}^\dagger, \hat{q}^\dagger, \hat{\mathfrak{p}}$), we can similarly describe the part of the MAP that corresponds to $X^*$ until its first passage from $(-\infty,0)$ to $(0,\infty)$, when issued from a negative value. 
\smallskip

\noindent  Thanks to the piecewise evolution of $X^*$, considered at each crossing of the origin, we can thus conclude that it is a rssMp with underlying MAP whose matrix exponent is given by
\begin{align*}
&\mathbf{\Psi}^*(\theta) \\
&=- \left[
\begin{array}{cc}
\Psi^\dagger(\theta) - q^\dagger + q^\dagger \mathfrak{p}(1 - \mathbb{E}_1[(|X_{\tau^-_0}|/ X_{\tau^-_0-})^{\i\theta}]) & 0\\
0 &\hat{\Psi}^\dagger(\theta) - \hat{q}^\dagger + \hat{q}^\dagger\hat{ \mathfrak{p}}(1 - \hat{\mathbb{E}}_1[(|X_{\tau^-_0}|/ X_{\tau^-_0-})^{\i\theta}])
\end{array}\right]\\
&\hspace{2cm}+\left[\begin{array}{cc}
-q^\dagger (1-\mathfrak{p}) &q^\dagger (1-\mathfrak{p}) \mathbb{E}_1[(|X_{\tau^-_0}|/ X_{\tau^-_0-})^{\i\theta}] \\
\hat{q}^\dagger (1-\hat{\mathfrak{p}})\hat{\mathbb{E}}_1[(|X_{\tau^-_0}|/ X_{\tau^-_0-})^{\i\theta}]&-\hat{q}^\dagger (1-\hat{\mathfrak{p}}) \\
\end{array}
\right]\\
&= \left[
\begin{array}{cc}
-\Psi^\dagger(\theta) + q^\dagger \mathfrak{p}\mathbb{E}_1[(|X_{\tau^-_0}|/ X_{\tau^-_0-})^{\i\theta}] & q^\dagger (1-\mathfrak{p}) \mathbb{E}_1[(|X_{\tau^-_0}|/ X_{\tau^-_0-})^{\i\theta}] \\
\hat{q}^\dagger (1-\hat{\mathfrak{p}})\hat{\mathbb{E}}_1[(|X_{\tau^-_0}|/ X_{\tau^-_0-})^{\i\theta}] &-\hat{\Psi}^\dagger(\theta) + \hat{q}^\dagger\hat{ \mathfrak{p}} \hat{\mathbb{E}}_1[(|X_{\tau^-_0}|/ X_{\tau^-_0-})^{\i\theta}]
\end{array}\right]
\end{align*}
where $\hat{\mathbb{P}}$ is the law of $-X$.

\noindent Finally, we consider  the case of $Y^\natural$ and recall that $\rho=\frac{1}{2}$. 
The analysis is the same as in the case of $Y^*$ except that the compound Poisson process $\xi^\circ$ there gets replaced with a different compound Poisson process $\xi^\bullet$.
From the path analysis of $Y^\natural$ the jumps of $\xi^\bullet$ have the distribution of  $\log (X_{\sigma_0}/ X_{\tau^-_0-})$, while its rate  is $q^\dagger\mathfrak{q}$. The  characteristic function of $\log (X_{\sigma_0}/ X_{\tau^-_0-})$ follows from the analysis in \cite[Proposition~4.2]{Watson-censored}: 
\[
\mathbb{E}_1\left[\left(\frac{X_{\sigma_0}}{X_{\tau^-_0-}}\right)^{\i \theta}\right]=\frac{\Gamma(1-\frac{\alpha}{2}+\i\theta)\Gamma(\frac{\alpha}{2}-\i\theta)\Gamma(1+\i\theta)\Gamma(\alpha-\i\theta)}{\Gamma(\frac{\alpha}{2})\Gamma(1-\frac{\alpha}{2}) \Gamma(\alpha)},\quad \theta\in\mathbb{R}.
\]
This completes the proof.
\end{proof}

From the above proof we can recover the result of \cite{budd}. Indeed, from \cite[Corollary~1]{cc}
\[
\Psi^\dagger(\theta)=\frac{\Gamma(\alpha-\i\theta)\Gamma(1+\i\theta)}{\Gamma(\alpha\hat{\rho}-\i\theta)\Gamma(1-\alpha\hat{\rho}+\i\theta)},\quad z\in \mathbb{R},
\]
 in particular $q^\dagger =\Psi^\dagger(0)=\Gamma(\alpha)/(\Gamma(\alpha\hat{\rho})\Gamma(1-\alpha\hat{\rho}))$.
It now follows from \eqref{psistart} that 
\begin{align*}
\Psi^*(\theta)& =\Psi^\dagger(\theta) - q^\dagger \mathfrak{p} \mathbb{E}_1[(|X_{\tau^-_0}|/ X_{\tau^-_0-})^{\i\theta}]  
\\
&=\frac{\Gamma(\alpha-\i\theta)\Gamma(1+\i\theta)}{\Gamma(\alpha\hat{\rho}-\i\theta)\Gamma(1-\alpha\hat{\rho}+\i\theta)} - \frac{\Gamma(\alpha)}{\Gamma(\alpha\hat{\rho})\Gamma(1-\alpha\hat{\rho})}\mathfrak{p} \frac{\Gamma(\alpha-\i\theta)\Gamma(1+\i\theta)}{\Gamma(\alpha)} \\
&=\frac{\Gamma(\alpha-\i\theta)\Gamma(1+\i\theta)}{\pi}\left[\sin(\pi(\alpha\hat{\rho} -\i\theta))-\mathfrak{p}\sin(\pi\alpha\hat{\rho})\right], \qquad \theta\in\mathbb{R},
\end{align*}
which agrees with \eqref{eq:budd}. Likewise we can deduce the explicit form of $\Psi^\natural$ reported in \eqref{psiglued}.

\smallskip

\noindent  In a similar spirit, we can proceed to compute the matrix exponent $\mathbf{\Psi}^*$ to obtain
\begin{align}
&\mathbf{\Psi}^*(\theta)\notag\\
&= \frac{\Gamma(\alpha-\i\theta)\Gamma(1+\i\theta)}{\pi}
\left[
\begin{array}{cc}
\mathfrak{p}\sin(\pi\alpha\hat{\rho})-\sin(\pi(\alpha\hat{\rho} -\i\theta)) & (1-\mathfrak{p})\sin(\pi\alpha\hat{\rho})\\
(1-\hat{\mathfrak{p}})\sin(\pi\alpha\rho)& \hat{\mathfrak{p}}\sin(\pi\alpha\rho) -\sin(\pi(\alpha\rho-\i\theta))
\end{array}
\right],
\end{align}
for  $\theta\in \mathbb{R}$. For $\mathfrak{p}=\hat{\mathfrak{p}}=0$ this reduces to the matrix exponent of a stable process killed on hitting $0$ as discussed in \cite{deep1}. We conclude this discussion with the following corollary which identifies $\mathbf{\Psi}$ in a form that is similar to the  double hypergeometric analogoue.
\begin{corollary}
For $\theta\in\mathbb{R}$, 
\begin{align*}
&\mathbf{\Psi}^*(\theta)\\
&=-2^\alpha\left[
\begin{array}{cc}
 \frac{\Gamma \left(\frac{1+\i\theta}{2}\right)\Gamma\left(\frac{2+\i\theta}{2}\right)}{\Gamma\left( \frac{\sigma+b+\i\theta}{2}\right)\Gamma\left(\frac{\sigma-b+\i\theta+2}{2}\right)}\frac{\Gamma \left(\frac{\alpha-\i\theta}{2}\right)\Gamma\left(\frac{1+\alpha-\i\theta}{2}\right)}{\Gamma\left( \frac{b-\sigma-\i\theta}{2}\right)\Gamma\left( \frac{2-\sigma-b-\i\theta}{2}\right)} &
-\frac{\Gamma \left(\frac{1+\i\theta}{2}\right)\Gamma\left(\frac{2+\i\theta}{2}\right)}{\Gamma\left( \frac{\sigma+b}{2}\right)\Gamma\left(\frac{\sigma-b+2}{2}\right)}\frac{\Gamma \left(\frac{\alpha-\i\theta}{2}\right)\Gamma\left(\frac{1+\alpha-\i\theta}{2}\right)}{\Gamma\left( \frac{b-\sigma}{2}\right)\Gamma\left( \frac{2-\sigma-b}{2}\right)}\\
&\\
-\frac{\Gamma \left(\frac{1+\i\theta}{2}\right)\Gamma\left(\frac{2+\i\theta}{2}\right)}{\Gamma\left( \frac{\hat{\sigma}+\hat{b}}{2}\right)\Gamma\left(\frac{\hat{\sigma}-\hat{b}+2}{2}\right)}\frac{\Gamma \left(\frac{\alpha-\i\theta}{2}\right)\Gamma\left(\frac{1+\alpha-\i\theta}{2}\right)}{\Gamma\left( \frac{\hat{b}-\hat{\sigma}}{2}\right)\Gamma\left( \frac{2-\hat{\sigma}-\hat{b}}{2}\right)}
&\frac{\Gamma \left(\frac{1+\i\theta}{2}\right)\Gamma\left(\frac{2+\i\theta}{2}\right)}{\Gamma\left( \frac{\hat{\sigma}+\hat{b}+\i\theta}{2}\right)\Gamma\left(\frac{\hat{\sigma}-\hat{b}+\i\theta+2}{2}\right)}\frac{\Gamma \left(\frac{\alpha-\i\theta}{2}\right)\Gamma\left(\frac{1+\alpha-\i\theta}{2}\right)}{\Gamma\left( \frac{\hat{b}-\hat{\sigma}-\i\theta}{2}\right)\Gamma\left( \frac{2-\hat{\sigma}-\hat{b}-\i\theta}{2}\right)}
\end{array}
\right],
\end{align*}
where, as before, 
$
\sigma=-\alpha\hat{\rho}+1/2$  and $b:=\arccos(\mathfrak{p}\cos(\pi\sigma))/{\pi}$
and $\hat\sigma$ and $\hat{b}$ have the obvious meaning.\qed
\end{corollary}

\noindent We note that the condition $\{\mathfrak{p},\hat{\mathfrak{p}}\}\subset [0,1)$ ensures that $J$ is irreducible. If $\mathfrak{p}=1$, then, albeit possibly only after the first crossing of the origin (in the case that $X^*$ is issued from a negative value), we are basically in the case of the pssMp $Y^*$.  An analogous statement can be made for $\hat{\mathfrak{p}}=1$.
\smallskip

\noindent For real $\theta$ in some neighbourhood of $0$, the explicit expression for $\mathbf{\Psi}^*(-\i \theta)$ can be used to determine its Perron-Frobenius eigenvalue $\chi(\theta)$ \cite[Paragraph~XI.2C]{asmussen}. From the aforesaid reference, it is known that $\chi(0)=0$
and that, when it exists, $\chi'(0)$ dictates the long-term behavior of $(\xi,J)$, the underlying MAP of $X^*$, in the sense that a.s. $\lim_{t\to\infty}\xi_t/t=\chi'(0)$ \cite[Corollary~XI.2.8]{asmussen}. Moreover, $\xi$ drifts to $\infty$, oscillates, or drifts to $-\infty$, according as to whether $\chi'(0)$  is valued  $>0$, $=0$ or $<0$ \cite[Proposition~XI.2.10]{asmussen}. By the Lamperti--Kiu transform, this corresponds to $\vert X^*\vert$ drifting to $\infty$, oscillating between $0$ and $\infty$ and hitting $0$ continuously, respectively. 
 \smallskip

\noindent 
An explicit computation yields 
$$\chi'(0)=\Gamma(\alpha)\frac{(1-\hat{\mathfrak{p}})\sin(\pi\alpha\rho)\cos(\pi\alpha\hat{\rho})+(1-\mathfrak{p})\sin(\pi\alpha\hat{\rho})\cos(\pi\alpha\rho)}{(1-\hat{\mathfrak{p}})\sin(\pi\alpha\rho)+(1-\mathfrak{p})\sin(\pi\alpha\hat{\rho})}.$$
From this we note that when $\mathfrak{p}=\hat{\mathfrak{p}}$ (in particular when both are equal to $0$ and $X^*=(X_t\mathbf{1}_{(t<\tau^-_0)}, t\geq 0)$), one obtains $\chi'(0)(\sin(\pi\alpha\rho)+\sin(\pi\alpha\hat{\rho}))=\Gamma(\alpha)\sin(\pi\alpha)$ (which is consistent with the polarity of $0$ for $X$ when $\alpha\in (0,1]$ and the fact that $X$ hits $0$ when $\alpha\in (1,2)$). Furthermore, whether or not $\mathfrak{p}=\hat{\mathfrak{p}}$, $\chi'(0)>0$ whenever $\{\alpha\rho,\alpha\hat{\rho}\}\subset (0,\frac{1}{2})$, in particular whenever $\alpha< \frac{1}{2}$, but in general the sign of $\chi'(0)$ will depend on $\alpha,\rho,\mathfrak{p},\hat{\mathfrak{p}}$ in a non-trivial ``balancing'' way. For instance, if $\hat{\rho}\alpha\in (\frac{1}{2},1)$, $\hat{\mathfrak{p}}$ is sufficiently close to $0$, while $\mathfrak{p}$ is sufficiently close to $1$, then $X^*$ will hit $0$ continuously, even though one may  also have $\alpha< 1$, in which case $(X_t\mathbf{1}_{(t<\tau^-_0)}, t\geq 0)$ will not be absorbed continuously at the origin.  More generally, we have

\begin{proposition}\label{proposition:hit-0-cts-general}
$X^*$  hits zero continuously iff $(1-\hat{\mathfrak{p}})\sin(\pi\alpha\rho)\cos(\pi\alpha\hat{\rho})+(1-\mathfrak{p})\sin(\pi\alpha\hat{\rho})\cos(\pi\alpha\rho)<0$, equivalently $\sin(\pi\alpha)<\hat{\mathfrak{p}}\sin(\pi\alpha\rho)\cos(\pi\alpha\hat{\rho})+\mathfrak{p}\sin(\pi\alpha\hat{\rho})\cos(\pi\alpha\rho)$; for this condition to prevail it is necessary (but not sufficient) that $\alpha\hat{\rho}>\frac{1}{2}$ or $\alpha\rho>\frac{1}{2}$. 
\end{proposition}
\begin{proof}
The denominator of $\chi'(0)$ is always strictly positive because $0<\alpha\rho<1$ and $0<\alpha\hat{\rho}<1$ (see for instance \cite[p.~3]{deep1} for the identification of the admissible pairs $(\alpha,\rho)$ under our standing assumptions). Therefore the condition for $X^*$ to hit $0$ continuously, i.e. $\chi'(0)<0$, becomes the stated one, with its equivalent form following upon using the addition formula for the sine. For the final claim one notes that $\cos(\pi\alpha\hat{\rho})$ and $\cos(\pi\alpha\rho)$ are both nonnegative unless  $\alpha\hat{\rho}>\frac{1}{2}$ or $\alpha\rho>\frac{1}{2}$. %The necessary condition is clear from the first form.% of the) condition.  %$(1-\hat{\mathfrak{p}})\sin(\pi\alpha\rho)\cos(\pi\alpha\hat{\rho})+(1-\mathfrak{p})\sin(\pi\alpha\hat{\rho})\cos(\pi\alpha\rho)<0$. It is clear that this can only happen when $\alpha\hat{\rho}>\frac{1}{2}$ or $\alpha\rho>\frac{1}{2}$. The equivalent condition stated in the proposition follows by the addition formula for the sine.
\end{proof}

\section{Proofs of main results}\label{section:proofs}

\subsection{Proof of Theorem~\ref{proposition}}\label{subsection:proof-of-main-thm}
\noindent Let $\alpha<\beta$ and $\gamma<\delta$ be strictly positive real numbers and assume $\{\alpha-\beta,\gamma-\delta,\gamma-\alpha,\delta-\alpha, \gamma-\beta,\delta-\beta\}\cap \ZZ=\emptyset$, i.e. assume  $\gamma+\ZZ_{\leq 0}$, $\delta+\ZZ_{\leq 0}$, $\alpha+\ZZ_{\leq 0}$ and $\beta+\ZZ_{\leq 0}=\emptyset$ are pairwise disjoint. Then  we claim that $\{\gamma,\delta\}+ \ZZ_{\leq 0}$ interlaces with $\{\alpha,\beta\}+\ZZ_{\leq 0}$ if and only if either
\begin{enumerate}[(I)]
\item \label{interlace:I} there is a $k\in \mathbb{N}_0$ with $\beta-k-1<\gamma<\alpha<\delta-k<\beta-k$, or
\item\label{interlace:II} there is a $k\in \mathbb{N}$ with $\delta-k<\alpha<\gamma<\beta-k<\delta-k+1$.
\end{enumerate}
%Moreover, when these conditions prevail, then $\gamma+\delta<\alpha+\beta<\gamma+\delta+1$.
\smallskip

\noindent The conditions are clearly sufficient. For necessity, assume the interlacing property. Clearly we must have $\alpha<\delta<\beta$ and then either (i) $\gamma<\alpha<\delta<\beta$ or else (ii) $\alpha <\gamma<\delta<\beta$. Assume now (i). Then $\beta-1<\delta$. If even $\beta-1<\alpha$, then necessarily $\beta-1<\gamma$, and we are done. Otherwise $\gamma<\alpha<\beta-1<\delta<\beta$. But then we must have $\gamma<\alpha<\delta-1<\beta-1<\delta<\beta$. An inductive argument allows to conclude. Assume (ii). We must have $\alpha <\gamma<\beta-1<\delta<\beta$. If $\delta-1<\alpha$ then we are done. Otherwise $\alpha<\delta-1$ and necessarily $\alpha <\gamma<\beta-2<\delta-1<\beta-1<\delta<\beta$. An inductive argument allows to conclude.
Clearly \ref{interlace:I} and \ref{interlace:II} are exclusive. %The final claim of the previous paragraph is immediate.
\smallskip

\noindent Assume now further that the interlacing property holds. The above characterization implies that $-\alpha-\beta<-\gamma-\delta<-\alpha-\beta+1 $, i.e. $\gamma+\delta<\alpha+\beta<\gamma+\delta+1$. \smallskip

\noindent Next, via Euler's infinite product formula for the gamma function, we may write, for $z\in \mathbb{C}$ where the expression is defined, 
$$\Phi(z):=\frac{\Gamma(z+\alpha)\Gamma(z+\beta)}{\Gamma(z+\gamma)\Gamma(z+\delta)}=\frac{\Gamma(\alpha)\Gamma(\beta)}{\Gamma(\gamma)\Gamma(\delta)}\prod_{n=0}^\infty \frac{(1-\frac{z}{-n-\gamma})(1-\frac{z}{-n-\delta})}{(1-\frac{z}{-n-\alpha})(1-\frac{z}{-n-\beta})}.$$
Then $\Phi$ is a real meromorphic function, and it follows from \cite[Theorem 27.2.1]{levin} and from the interlacing condition together with $\{\alpha,\beta,\gamma,\delta\}\subset (0,\infty)$ that it maps the open upper complex half-plane into itself, i.e. it is a Pick function, and hence a non-constant completely monotone Bernstein function. It follows furthermore \cite[Problem~27.2.1]{levin} that $\Phi$ admits the representation, for some  $\{\omega_1,\omega_2\}\subset \ZZ\cup \{-\infty,\infty\}$, $\omega_1\leq\omega_2$, and at least for $z$ in the open upper complex half-plane,
\begin{equation}\label{eq:H}
\Phi(z)=az+b+\frac{d}{z}+{\sum}_{n=\omega_1}^{\omega_2} c_n\left[\frac{1}{a_n-z}-\frac{1}{a_n}\right],
\end{equation} where $\{a,-d\}\subset [0,\infty)$, $b\in \mathbb{R}$, all the $c_n$ are nonnegative, the sequence $a_n$ (where $n\in\{\omega_1, \omega_1+1, \ldots, \omega_2\}$)  consists of non-zero real numbers and is strictly increasing, finally $\sum_{n=\omega_1}^{\omega_2}{c_n}/{a_n^2}<\infty$. Then the $a_n$ and $0$ must exhaust the poles of $\Phi$ and the $c_n$ must be the respective negatives of the residua of $\Phi$ at these poles, which can easily be computed: for $k\in \mathbb{N}_0$, 
\[
\mathrm{Res}(\Phi,-\alpha-k)=(-1)^k\frac{\Gamma(\beta-\alpha-k)}{k!\Gamma(\gamma-\alpha-k)\Gamma(\delta-\alpha-k)}
\]
and 
\[
\mathrm{Res}(\Phi,-\beta-k)=(-1)^k\frac{\Gamma(\alpha-\beta-k)}{k!\Gamma(\gamma-\beta-k)\Gamma(\delta-\beta-k)}.
\]
 The coefficient $d$ equals $0$ because $\Phi$ has no pole at $0$ (as $\{\alpha,\beta\}\subset (0,\infty)$). By continuity of $\Phi$ and dominated convergence for the right-hand side, the equality \eqref{eq:H} extends to $[0,\infty)$. Then by dominated convergence and by Stirling's formula for the gamma function, using $\alpha+\beta<\gamma+\delta+1$, we see that $a=\lim_{z\to \infty}{\Phi(z)}/{z}=0$ (limit over $(0,\infty)$).
\smallskip

\noindent On the other hand, from the definition of the hypergeometric  function ${_2F_1}$, for $s\in (0,\infty)$, \footnotesize
$$f(s)=\sum_{k=0}^\infty (-1)^{k+1}\frac{\Gamma(\beta-\alpha-k)}{k!\Gamma(\gamma-\alpha-k)\Gamma(\delta-\alpha-k)}e^{-(\alpha+k)s}+\sum_{k=0}^\infty (-1)^{k+1}\frac{\Gamma(\alpha-\beta-k)}{k!\Gamma(\gamma-\beta-k)\Gamma(\delta-\beta-k)}e^{-(\beta+k)s}.$$\normalsize
A direct computation using Tonelli-Fubini (note that we know a priori from the preceding that $\sum_{n=\omega_1}^{\omega_2}{c_n}/{a_n^2}<\infty$), then reveals that for $z\in \mathbb{C}$ with $\Re(z)\geq 0$,
\begin{align}
H(z)&:=\frac{\Gamma(\alpha)\Gamma(\beta)}{\Gamma(\gamma)\Gamma(\delta)}+\int_0^\infty (1-e^{-z s})f(s)ds\nonumber\\
&=\frac{\Gamma(\alpha)\Gamma(\beta)}{\Gamma(\gamma)\Gamma(\delta)}+\sum_{k=0}^\infty (-1)^{k+1}\frac{\Gamma(\beta-\alpha-k)}{k!\Gamma(\gamma-\alpha-k)\Gamma(\delta-\alpha-k)}\left(\frac{1}{\alpha+k}-\frac{1}{\alpha+k+z}\right)\nonumber\\
&\phantom{=}+\sum_{k=0}^\infty (-1)^{k+1}\frac{\Gamma(\alpha-\beta-k)}{k!\Gamma(\gamma-\beta-k)\Gamma(\delta-\beta-k)}\left(\frac{1}{\beta+k}-\frac{1}{\beta+k+z}\right).\label{eq:Phi}
\end{align}
\smallskip

\noindent Comparing \eqref{eq:H} and \eqref{eq:Phi}, we see that certainly $\Phi=H$ on $[0,\infty)$ and hence by analyticity and continuity on the whole of the closed right complex half-plane, which identifies the subordinator corresponding to $\Phi$ as being pure-jump with L\'evy density $f$. Incidentally,  because all of the $c_n$ must  be nonnegative, we also obtain that the two series appearing in the above representation of $f$ are each of positive terms only. %By the identification $\Phi\vert_{[0,\infty)}=H\vert\vert_{[0,\infty)}$ o 
% then $f$ is completely monotone (one can differentiate under the summation sign). 
\smallskip

\noindent %Finiteness of $\int_0^\infty (1\land s)f(s)ds<\infty$ comes directly from the finiteness of $\Phi=H$ on $[0,\infty)$. 
The assumption $\{\alpha,\beta,\gamma,\delta\}\subset (0,\infty)$ ensures that $\Phi(0)>0$, i.e. the subordinator corresponding to $\Phi$ is killed. By monotone convergence $\Phi$ has infinite or finite activity according as $\int_0^\infty f(s)ds=\lim_{z\to \infty}\Phi(z)=\infty$ or $<\infty$, and Stirling's formula for the gamma function together with $\gamma+\delta<\alpha+\beta$ guarantees that it is the former that is taking place here. 
\smallskip

\noindent Further, as a complete Bernstein function, the potential measure associated to $\Phi$ has a density with respect to Lebesgue measure; cf. \cite[Remark~11.6]{bernstein}
and \cite[Lemma~2.3]{vondracek}. In order to check that \eqref{eq:potential} gives the potential density, write it out using the definition of ${_2F_1}$: for $x\in (0,\infty)$,\footnotesize
$$u(x)=\sum_{k=0}^\infty (-1)^{k}\frac{\Gamma(\delta-\gamma-k)}{k!\Gamma(\alpha-\gamma-k)\Gamma(\beta-\gamma-k)}e^{-(\gamma+k)x}+\sum_{k=0}^\infty (-1)^{k}\frac{\Gamma(\gamma-\delta-k)}{k!\Gamma(\alpha-\delta-k)\Gamma(\beta-\delta-k)}e^{-(\delta+k)x}.$$\normalsize
Note that the two series consist of positive terms owing to the characterizations \ref{charact:i}-\ref{charact:ii} and the properties of the sign of the gamma function on the real line. By Tonelli's theorem we compute then the Laplace transform $\hat{u}$ of  \eqref{eq:potential}, \footnotesize
$$\hat{u}(z)=\sum_{k=0}^\infty (-1)^{k}\frac{\Gamma(\delta-\gamma-k)}{k!\Gamma(\alpha-\gamma-k)\Gamma(\beta-\gamma-k)}\frac{1}{\gamma+k+z}+\sum_{k=0}^\infty (-1)^{k}\frac{\Gamma(\gamma-\delta-k)}{k!\Gamma(\alpha-\delta-k)\Gamma(\beta-\delta-k)}\frac{1}{\delta+k+z}$$ ($z\in \mathbb{C}$, $\Re(z)\geq 0$) \normalsize and check that it coincides with $1/\Phi$. 
\smallskip

\noindent To this end, let $z\in \mathbb{C}$, $\Re(z)\geq 0$; we verify that $\hat{u}(z)=\frac{1}{\Phi(z)}$.
But $\Phi$ being a complete Bernstein function implies \cite[Proposition~7.1]{bernstein} that $\frac{\mathrm{id}_{(0,\infty)}}{\Phi}$ is itself a complete Bernstein function. Then \cite[Remark~6.4]{bernstein} guarantees that one can write $$\frac{1}{\Phi(z)}=\frac{\Gamma(\gamma+z)\Gamma(\delta+z)}{\Gamma(\alpha+z)\Gamma(\beta+z)}=\frac{a}{z}+b+\int_{(0,\infty)}\frac{1}{z +t}\eta(\D t),\quad z\in (0,\infty),$$ where $\{a,b\}\subset [0,\infty)$ and $\eta$ is a measure on $\mathcal{B}_{(0,\infty)}$ such that $\int (1+t)^{-1}\eta(\D t)<\infty$. Clearly we must have $b=0$ because $\lim_{z\to\infty}{1}/{\Phi}(z)=0$. The measure  $\eta$ is supported by $\{\delta,\gamma\}+\mathbb{N}_0$, which follows by appealing to e.g. \cite[Corollary~6.3]{bernstein}.  Moreover,   $a=0$ because $\lim_{z\downarrow 0}{1}/{\Phi(z)}<\infty$. By analytic extension, the masses of $\eta$ are identified via identifying the residues of $\Phi(z)$, which concludes the argument.
\smallskip

\noindent When, ceteris paribus,  the interlacing condition fails, it follows from its representation as an infinite product given above, and from the characterization of Pick maps of \cite[Theorem 27.2.1]{levin} that the function $\Phi(z)$ is not Pick,  hence  also not a non-constant complete Bernstein function. 

\noindent The final assertion of the theorem follows by recalling  that complete Bernstein functions are closed under pointwise limits \cite[Corollary~7.6(ii)]{bernstein}.
\smallskip

\qed

\noindent Let us make some remarks concerning the above proof.  

\smallskip

\noindent First, to establish the complete Bernstein property of \eqref{doublebeta} under the ``interlacing conditions''  we could have relied directly on the results of \cite{meromorphic} (see Lemma~1 and the discussion following its proof therein). However, we wanted to establish the precise, not just sufficient, conditions under which the ``non-constant complete Bernstein'' property obtains, which is why we have provided above a direct and detailed study of the Pick property of \eqref{doublebeta}. Besides, this has the advantage of keeping our paper more self-contained. 

\smallskip

\noindent Secondly, the expression \eqref{eq:density} for the L\'evy density of course does not just ``fall out from the sky''; rather it is got by differentiating twice the inverse Laplace transform of 
\[
z\mapsto \frac{1}{z^2}\frac{\Gamma(\alpha+z)\Gamma(\beta+z)}{\Gamma(\gamma+z)\Gamma(\delta+z)}
\]
 (cf. \cite[Eq.~(3.4)]{bernstein}), and the latter in turn is obtained (at least on a pro forma level) by the Heaviside residue expansion \cite[Theorem~10.7(c)]{heaviside}. A similar remark pertains to \eqref{eq:potential}: one knows \cite[Eq.~(5.20)]{bernstein} that the Laplace transform of $U$ is $1/\Phi$, so that at least on a pro forma level one can Laplace invert via the Heaviside expansion.

\smallskip

\noindent Finally, we recover from the proof of Theorem~\ref{proposition} the non-obvious identities\footnotesize
$$
\frac{\Gamma(\alpha+z)\Gamma(\beta+z)}{\Gamma(\gamma+z)\Gamma(\delta+z)}=\frac{\Gamma(\alpha)\Gamma(\beta)}{\Gamma(\gamma)\Gamma(\delta)}+\sum_{k=0}^\infty (-1)^{k+1}\frac{\Gamma(\beta-\alpha-k)}{k!\Gamma(\gamma-\alpha-k)\Gamma(\delta-\alpha-k)}\left(\frac{1}{\alpha+k}-\frac{1}{\alpha+k+z}\right)
$$
$$+\sum_{k=0}^\infty (-1)^{k+1}\frac{\Gamma(\alpha-\beta-k)}{k!\Gamma(\gamma-\beta-k)\Gamma(\delta-\beta-k)}\left(\frac{1}{\beta+k}-\frac{1}{\beta+k+z}\right),\quad z\in \mathbb{C}\backslash (\{-\alpha,-\beta\}+\ZZ_{\leq 0})$$\normalsize
and\footnotesize
 $$\sum_{k=0}^\infty (-1)^{k}\frac{\Gamma(\delta-\gamma-k)}{k!\Gamma(\alpha-\gamma-k)\Gamma(\beta-\gamma-k)}\frac{1}{\gamma+k+z}+\sum_{k=0}^\infty (-1)^{k}\frac{\Gamma(\gamma-\delta-k)}{k!\Gamma(\alpha-\delta-k)\Gamma(\beta-\delta-k)}\frac{1}{\delta+k+z}$$
$$=\frac{\Gamma(\gamma+z)\Gamma(\delta+z)}{\Gamma(\alpha+z)\Gamma(\beta+z)},\quad z\in \mathbb{C}\backslash (\{-\delta,-\gamma\}+\ZZ_{\leq 0}).$$\normalsize

\subsection{Proof of Propositon \ref{doubledensity}}\phantom{phantom}

\noindent  \ref{further:characteristics}. The meromorphic property can be seen from \cite[Theorem~1(v)]{meromorphic}, %i.e. one checks that $z\mapsto \psi(z)/z$ maps the open upper complex half-plane into itself. This follows easily from  \cite[Theorem 27.2.1]{levin}, 
using simply the representation of the real meromorphic function $\psi:=(z\mapsto -\Psi(-\i z))$ as an infinite product, that was essentially procured already in the proof of Theorem~\ref{proposition} (via Euler's infinite product formula for the gamma function).
 Then, according to \cite[p. 1105]{meromorphic}, this means a priori that the L\'evy measure $\Pi$ of $\xi$ admits a density $\pi$ of the form $$\frac{\pi(dx)}{dx}=\mathbbm{1}_{(0,\infty)}(x)\sum_{n\in \mathbb{N}}a_n\rho_n e^{-\rho_n x}+\mathbbm{1}_{(-\infty,0)}(x)\sum_{n\in \mathbb{N}}\hat{a}_n\hat{\rho}_n e^{\hat{\rho}_nx},\quad x\in \mathbb{R},$$ where all the coefficients are nonnegative and the sequences $(\rho_n)_{n\in \mathbb{N}}$ and $(\hat{\rho}_n)_{n\in \mathbb{N}}$ are strictly positive and strictly increasing to $\infty$. Furthermore, from  \cite[Eq.~(8)]{meromorphic}, we have that for some $\sigma^2\in[0,\infty)$ and $\mu\in \mathbb{R}$, \footnotesize $$-\frac{\Gamma(\alpha-z)\Gamma(\beta-z)}{\Gamma(\gamma-z)\Gamma(\delta-z)}\frac{\Gamma(\hat{\alpha}+z)\Gamma(\hat{\beta}+z)}{\Gamma(\hat{\gamma}+z)\Gamma(\hat{\delta}+z)}=\psi(z)=\Psi(0)+\frac{1}{2}\sigma^2z^2+\mu z+z^2\sum_{n\in \mathbb{N}}\left[\frac{a_n}{\rho_n(\rho_n-z)}+\frac{\hat{a}_n}{\hat{\rho}_n(\hat{\rho}_n+z)}\right].$$ \normalsize Dividing in the previous display by $z^2$, and considering the residua, we find: %$\frac{1}{2}\sigma^2=\lim_{z\downarrow 0}$
that $-\frac{a_n}{\rho_n}=\mathrm{Res}(\psi(z)/z^2;z=\rho_n)$, i.e. $a_n\rho_n=\mathrm{Res}(-\psi;\rho_n)=-\mathtt{B}(\hat{\alpha}, \hat{\beta}, \hat{\gamma}, \hat{\delta};\rho_n)\mathrm{Res}(\mathtt{B}(\alpha, \beta, \gamma, \delta;\cdot),-\rho_n)$; similarly that $\frac{\hat{a}_n}{\hat{\rho}_n}=\mathrm{Res}(\psi(z)/z^2;z=-\hat{\rho}_n)$, i.e. $\hat{a}_n\hat{\rho}_n=-\mathtt{B}(\alpha, \beta, \gamma, \delta;\hat{\rho}_n)\mathrm{Res}(\mathtt{B}(\hat{\alpha}, \hat{\beta}, \hat{\gamma}, \hat{\delta};\cdot),-\hat{\rho}_n)$; and that the $\rho_n$, $n\in \mathbb{N}$, must run over $\{\alpha,\beta\}+\mathbb{N}_0$, while the $\hat{\rho}_n$, $n\in \mathbb{N}$, must run over $\{\hat{\alpha},\hat{\beta}\}+\mathbb{N}_0$. Computing the residua and simplifying yields the expression for the L\'evy density.

\ref{further:characteristics:Gaussian}. By the asymptotic properties of the gamma function \cite[Formula~8.328.1]{tables} we see that for $\theta\in  \mathbb{R}$, $\vert\Psi(\theta)\vert\sim \vert \theta\vert^{\alpha+\beta+\hat{\alpha}+\hat{\beta}-\gamma-\delta-\hat{\gamma}-\hat{\delta}}$, as $\vert \theta\vert\to \infty$. Now the claim follows using \cite[Proposition~I.2(i)]{bertoin}, since we know a priori that $\alpha+\beta\leq \gamma+\delta+1$ and $\hat{\alpha}+\hat{\beta}\leq \hat{\gamma}+\hat{\delta}+1$.
% $X$ has a Gaussian component iff $\alpha+\beta+\hat{\alpha}+\hat{\beta}-\gamma-\delta-\hat{\gamma}-\hat{\delta}=2$, i.e. $\alpha+\beta=\gamma+\delta+1$ and $\hat{\alpha}+\hat{\beta}=\hat{\gamma}+\hat{\delta}+1$, in which case the diffusion coefficient is equal to $\sigma^2=2$.

\ref{further:characteristics:lifetime}. This follows from  inspecting  whether or not the  Laplace exponents of the increasing and decreasing ladder heights subordinators of $\xi$ (which are known from Corollary~\ref{corollary}) vanish at $0$.

\subsection{Proof of Proposition \ref{proposition:mellin}}
\noindent As alluded to above, we will use the verification argument of \cite[Proposition~2]{hg-bis}. The idea there is to verify that the right hand side of \eqref{eq:mellin} has certain analytical properties so that it matches the unique solution of the functional equation that the left-hand side of \eqref{eq:mellin} must necessarily solve (cf. \cite{MZ}).

\smallskip

\noindent
The Laplace exponent $\psi_{c}(z): = \log \mathbf{E}[\exp(z\xi_1/{c})]$ is given by the map $z\mapsto \psi(z/{c})$, where $\psi$ is the Laplace exponent of $\xi$.  The conditions on the parameters ensure  that $\xi/{c}$ either drifts to $\infty$ or is killed (see Proposition~\ref{doubledensity}\ref{further:characteristics:lifetime}, and that it satisfies the so-called Cram\'er's condition,
 $\psi_{c}(-\hat{\gamma}{c})=\psi(-\hat{\gamma})=0$ and $\psi_{c}$ is finite on $(-\hat{\alpha}{c},0)$, which contains $-\hat{\gamma}{c}$. 
Furthermore, letting $\mathcal{M}(s)$ be the right-hand side of \eqref{eq:mellin}, we see immediately from the properties of the double gamma function that
 $\mathcal{M}$ is analytic and free of zeros in the strip $\Re(s)\in (0,1+\hat{\gamma}{c})$ and 
 $\mathcal{M}(1)=1$ and $\mathcal{M}(s+1)=-s\mathcal{M}(s)/\psi_{c}(-s)$ for $s\in (0,\hat{\gamma}{c})$.
\smallskip

\noindent It remains to check that 
$\lim_{\vert y\vert \to\infty}\vert \mathcal{M}(x+iy)\vert^{-1}e^{-2\pi \vert y\vert}=0$ uniformly in  $x\in (0,\hat{\gamma}{c}+1)$.
Lemma 1 in \cite{hg-bis} implies that, as $\vert y\vert\to\infty$, 
$$\ln\left\vert \frac{G(p+iy;{c})}{G(q+iy;{c})}\right\vert=\frac{p-q}{{c}}\Re[(q+iy)\ln(q+iy)]+o(\vert y\vert)=-\pi\frac{p-q}{2{c}}\vert y\vert+o(\vert y\vert),$$
 uniformly in bounded real $p$ and $q$. Note this was  observed in \cite[p. 44]{budd} albeit there  without the uniformity, which is needed for our purposes.
 At the same time, by Stirling's asymptotic formula for the gamma function \cite[Formula~8.327]{tables},  as $\vert y\vert\to \infty$, we have
$$ \ln\vert \Gamma(x+iy)\vert=\Re[(x+iy)(\ln(x+iy)-1)]+o(\vert y\vert)=-\frac{\pi}{2}\vert y\vert+o(\vert y\vert),$$
uniformly for  $x$ bounded in $ \mathbb{R}$.

\smallskip

\noindent
 Therefore, it follows that, uniformly for $x\in (0,\hat{\gamma}{c}+1)$, as $\vert y\vert\to \infty$, we have
 $$\ln\vert \mathcal{M}(x+iy)\vert=-\frac{\pi}{2}\left(1+\frac{1}{2}(\gamma+\delta-\alpha-\beta+\hat{\alpha}+\hat{\beta}-\hat{\gamma}-\hat{\delta})\right)\vert y\vert+o(\vert y\vert).$$ 
Because $\gamma+\delta\leq \alpha+\beta$ and $\hat{\alpha}+\hat{\beta}\leq \hat{\gamma}+\hat{\delta}+1$ (see Theorem~\ref{proposition}), which gives us that  $1+(\gamma+\delta-\alpha-\beta+\hat{\alpha}+\hat{\beta}-\hat{\gamma}-\hat{\delta})/2<4$, the proof is now complete.\qed

\section*{Acknowledgements}
AEK acknowledges financial support from the EPSRC (EP/P009220/1 \& EP/L002442/1).
MV acknowledges financial support from the Slovenian Research Agency (research programmes Nos. P1-0222 \& P1-0402). JCP acknowledges support from   CONACyT-MEXICO. JCP and MV are grateful for the kind hospitality of the University of Bath, where both  were on a sabbatical while this research was conducted, with JCP holding a David Parkin Visiting Professorship  at the same university.
We thank Alexey Kuznetsov for providing some very useful insight into Pick functions and an anonymous Referee for his/her careful review of our paper.

\bibliographystyle{plain}
\bibliography{Biblio-ricocheted}
\end{document}